\RequirePackage{fix-cm}
\documentclass[smallextended]{svjour3}       % onecolumn (second format)
\smartqed  % flush right qed marks, e.g. at end of proof
\usepackage{graphicx}
\usepackage{amsfonts}
\usepackage{amsmath}

\begin{document}

\title{Viscosity solutions of Hamilton-Jacobi equations for
neutral-type systems\thanks{This work is supported by a grant of the RSF no. 21-71-10070,\\ https://rscf.ru/project/21-71-10070/}
}


\author{Anton Plaksin}


\institute{A. Plaksin \at
N.N. Krasovskii Institute of Mathematics and Mechanics of the Ural Branch of the Russian Academy of Sciences, 16 S. Kovalevskaya Str., Yekaterinburg, 620108, Russia;\\
Ural Federal University, 19 Mira street, 620002 Ekaterinburg, Russia \\
\email{a.r.plaksin@gmail.com}
}

\date{Received: date / Accepted: date}

\maketitle

\begin{abstract}
The paper deals with path-dependent Hamilton-Jacobi equations with a coinvariant derivative which arise in investigations of optimal control problems and differential games for neutral-type systems in Hale's form. A~viscosity (generalized) solution of a Cauchy problem for such equations is considered. The existence, uniqueness, and consistency of the viscosity solution are proved. Equivalent definitions of the viscosity solution, including the definitions of minimax and Dini solutions, are obtained. Application of the results to an optimal control problem for neutral-type systems in Hale's form are given.

\end{abstract}

%49L25 Viscosity solutions to Hamilton-Jacobi equations in optimal control and differential games
%34K40 Neutral functional-differential equations

\section{Introduction}\label{intro}

The paper aims to develop the viscosity solution theory for path-dependent Hamilton-Jacobi (HJ) equations with a coinvariant derivative which arise in optimal control problems and differential games for neutral-type systems in Hale's form \cite{Hale_1977}.

Initially, the notion of coinvariant derivatives and path-dependent HJ equations with such derivatives were considered in \cite{Kim_1999} to describe infinitesimal properties of value functionals in optimal control problems for time-delay systems. Also, similar derivative notions are known such as Clio derivatives \cite{Aubin_Haddad_2002} and horizontal and vertical derivatives \cite{Dupire_2009} (the connection between these notions was addressed in~\cite{Gomoyunov_Lukoyanov_Plaksin_2021}).  As with HJ equations with partial derivatives, path-dependent HJ equations with coinvariant derivatives may not have a differentiable solution and various approaches to the definition of a generalized solution needed to be examined.

The minimax approach \cite{Subbotin_1980,Subbotin_1984,Subbotin_1995} and its application to differential games  \cite{Krasovskii_Subbotin_1988,Krasovskii_Krasovskii_1995} for time-delay systems were developed in \cite{Lukoyanov_2000,Lukoyanov_2003,Lukoyanov_2010a} (see also \cite{Bayraktar_Keller_2018} for an extension of these results to an infinite dimensional setting). Investigations of the viscosity approach \cite{Crandall_Lions_1983,Crandall_Evans_Lions_1984} to the generalized solution definition for path-dependent HJ equations associated with time-delay systems can be roughly divided into two groups. In the first group \cite{Bayraktar_Keller_2018,Ekren_Touzi_Zhang_2016,Kaise_2015,Kaise_Kato_Takahashi_2018,Lukoyanov_2010b,Pham_Zhang_2014,Soner_1988,Zhou_2019}, modified definitions of the viscosity solution based on some parameterizations were researched. In the second group \cite{Plaksin_2020,Plaksin_2021,Zhou_2021}, more natural definitions of the viscosity solutions are used, but the path-dependent HJ equations are considered on the wider spaces of discontinuous functions.

In contrast to path-dependent HJ equations arising in optimal control and differential game problems for time-delay systems, path-dependent HJ equations associated with the problems for more general neutral-type systems in Hale's form have a new term. The fact that this term is not defined at all point of the functional space and discontinuous at some points of the domain do not allow us to apply the above results and require us to construct the new theory of generalized solutions.

The minimax solution theory for such equations and the application of this theory to differential games for neutral-type systems in Hale's form were investigated in  \cite{Gomoyunov_Plaksin_2019,Lukoyanov_Plaksin_2020,Lukoyanov_Plaksin_2020b,Plaksin_2021b}. The present paper develops the viscosity solution theory for such equations. Following \cite{Plaksin_2020,Plaksin_2021}, we consider the path-dependent HJ equations on the space of discontinuous functions and introduce the definition of a viscosity solution of a Cauchy problem for this equation. We prove that the viscosity solution exists and is unique (see Theorem \ref{teo:viscosity_solution}). We obtain additional equivalent definitions of the viscosity (generalized) solution including the definitions of minimax and Dini solutions (see Theorem \ref{teo:equivalent_solutions}). We also establish that a coinvariant differentiable solution of the Cauchy problem coincides with the viscosity solution and, on the other hand, the viscosity solution at the points of coinvariant differentiability satisfies the HJ equation (see Theorem \ref{teo:classical_and_viscosity_solutions}). Moreover, we consider an application of the results to an optimal control problem for a neutral-type system in Hale's form (see Theorem \ref{teo:application}).

The main idea of the proofs is to obtain the existence and uniqueness of the minimax solution based on results from \cite{Plaksin_2021b} and to establish the equivalence of the minimax and viscosity solutions using the scheme from \cite{Clarke_Ledyaev_1994,Subbotin_1993,Subbotin_1995} (see also \cite{Plaksin_2020}). However, there are several obstacles to implement that. Firstly, since the minimax solution from \cite{Plaksin_2021b} is defined on the space of Lipschitz continuous functions, we prove the certain Lipschitz continuous property of this solution and, using that, extend the minimax solution to the space of piecewise Lipschitz continuous functions (see Theorem \ref{teo:Lip_minimax_solution} and Theorem \ref{teo:minimax_solutions} $(a)\Leftrightarrow (b)$). Secondly, as noted above, the path-dependent HJ equations under considerations have the special term which is not defined on the whole space of piecewise Lipschitz continuous functions. Therefore, we introduce the definitions of a viscosity solution (see Definition \ref{def:viscosity_solution}) only on a certain subspace on which this term is defined. Such a definition seems one of the main features of the present paper, since usually the viscosity solution definitions for even more particular path-dependent HJ equations were considered on whole functional spaces. Besides, such a definition prevent us from applying the scheme from \cite{Subbotin_1993,Subbotin_1995} directly. Nonetheless, we overcome this obstacle by introducing an additional auxiliary definition of minimax solution on the space of continuously differentiable functions ($\mathrm{C}^1$-minimax solution in Definition \ref{def:C1_minimax_solution}). The fact that the $\mathrm{C}^1$-minimax solution coincides with the usual minimax solution (see Theorem \ref{teo:minimax_solutions} $(b)\Leftrightarrow (c)$) completes the proof of the equivalence of the minimax and viscosity solutions (see Theorem \ref{teo:equivalent_solutions}). Thirdly, the extension to the space of piecewise Lipschitz continuous functions does not allow us to expect a continuous solution, as opposed to \cite{Gomoyunov_Plaksin_2019,Lukoyanov_Plaksin_2020,Plaksin_2021b}. Nevertheless, the class of, generally speaking, discontinuous functionals, suggested in the paper, is suitable for obtaining the existence and uniqueness results within this class.

\section{Main results}

\subsection{Functional spaces}
Let $\mathbb R^n$ be the n-dimensional Euclidian space with the inner product $\langle \cdot, \cdot \rangle$ and the norm $\|\cdot\|$. A function $x(\cdot) \colon [a,b) \mapsto \mathbb R^n$ (or $x(\cdot) \colon [a,b] \mapsto \mathbb R^n$) is called piecewise Lipschitz continuous if there exist points $a = \xi_1 < \xi_2 < \ldots < \xi_k = b$ such that the function $x(\cdot)$ is Lipschitz continuous on the interval $[\xi_i,\xi_{i+1})$ for each $i \in \overline{1,k-1}$.
Note that such a function $x(\cdot)$ is right continuous on the interval $\xi \in [a,b)$ and has a finite left limit $x(\xi-0)$ for any $\xi \in (a,b]$.
Denote by $\mathrm{PLip}([a,b),\mathbb R^n)$ (or $\mathrm{PLip}([a,b],\mathbb R^n)$) the linear space of piecewise Lipschitz continuous functions $x(\cdot) \colon [a,b) \mapsto \mathbb R^n$ (or $x(\cdot) \colon [a,b] \mapsto \mathbb R^n$).

Let $\vartheta, h > 0$. Denote
\begin{equation*}
\mathrm{PLip} = \mathrm{PLip}([-h,0),\mathbb R^n),\quad
\mathrm{Lip} = \mathrm{Lip}([-h,0),\mathbb R^n),\quad
\mathrm{C}^1 = \mathrm{C}^1([-h,0),\mathbb R^n).
\end{equation*}
where $\mathrm{Lip}([-h,0),\mathbb R^n)$ and $\mathrm{C}^1([-h,0),\mathbb R^n)$ are the linear spaces of Lipschitz continuous and continuously differentiable functions $x(\cdot) \colon [a,b) \mapsto \mathbb R^n$. Denote
\begin{equation}\label{def:G_s}
\begin{array}{rl}
\mathrm{PLip}_* = & \big\{w(\cdot) \in \mathrm{PLip} \colon \text{there exists } \delta_w > 0 \text{ such that }\\[0.2cm]
& \ \,\, w(\cdot) \text{ is continuously differentiable on } [-h,-h+\delta_w] \big\}.
\end{array}
\end{equation}
Note that the following inclusions are valid:
\begin{equation}\label{space_inclusions}
\mathrm{C}^1 \subset \mathrm{Lip} \subset \mathrm{PLip},\quad \mathrm{C}^1 \subset \mathrm{PLip}_* \subset \mathrm{PLip}.
\end{equation}
For the sake of brevity, for any $w(\cdot) \in \mathrm{PLip}$, we denote
\begin{equation*}
\|w(\cdot)\|_1 = \int_{-h}^0 \|w(\xi)\| \mathrm{d} \xi,\quad \|w(\cdot)\|_\infty = \sup\limits_{\xi\in [-h,0)} \|w(\xi)\|,\quad w(-0) = w(0 - 0).
\end{equation*}
Without loss of generality of the results presented below, we can suppose the existence of $I \in \mathbb N$ such that $\vartheta =  I h$. Define the spaces
\begin{equation}\label{def:G_G_s}
\mathbb G = [0,\vartheta] \times \mathbb R^n \times \mathrm{PLip},\quad \mathbb G_* = \cup_{i=0}^{I-1} (i h, (i+1) h) \times \mathbb R^n \times \mathrm{PLip}_*.
\end{equation}

\subsection{Hamilton-Jacobi equation}

For each $(\tau,z,w(\cdot)) \in \mathbb G$, denote
\begin{equation*}
\begin{array}{rl}
\Lambda(\tau,z,w(\cdot)) = \big\{x(\cdot) \in \mathrm{PLip}([\tau-h,\vartheta],\mathbb R^n) \colon &  x(\tau) = z, \\[0.2cm]
& x(t) = w(t - \tau),\, t \in [\tau-h,\tau)\big\}.
\end{array}
\end{equation*}
Following \cite{Plaksin_2020,Plaksin_2021} (see also \cite{Kim_1999,Lukoyanov_2000}), a functional $\varphi \colon \mathbb G \mapsto \mathbb R$ is called coinvariantly (ci-) differentiable at a point $(\tau,z,w(\cdot)) \in \mathbb G$, $\tau < \vartheta$ if there exist $\partial^{ci}_{\tau,w}\varphi(\tau,z,w(\cdot)) \in \mathbb R$ and $\nabla_z\varphi(\tau,z,w(\cdot)) \in \mathbb R^n$ such that, for every $t \in [\tau,\vartheta]$, $y \in \mathbb R^n$, and $x(\cdot) \in \Lambda(\tau,z,w(\cdot))$, the relation below holds
\begin{equation}\label{def:ci-differentiability}
\begin{array}{c}
\varphi(t,y,x_t(\cdot)) - \varphi(\tau,z,w(\cdot))  =  (t - \tau) \partial^{ci}_{\tau,w}\varphi(\tau,z,w(\cdot)) \\[0.2cm]
+ \langle y - z, \nabla_z \varphi(\tau,z,w(\cdot)) \rangle + o(|t - \tau| + \|y - z\|),
\end{array}
\end{equation}
where $x_t(\cdot)$ denotes the function from $\mathrm{PLip}$ such that $x_t(\xi) = x(t + \xi)$, $\xi \in [-h,0)$ and the value $o(\delta)$ can depend on $x(\cdot)$ and $o(\delta)/\delta \to 0$ as $\delta \to +0$. Then $\partial^{ci}_{\tau,w}\varphi(\tau,z,w(\cdot))$ is called the ci-derivative of $\varphi$ with respect to $\{\tau,w(\cdot)\}$ and $\nabla_z \varphi(\tau,z,w(\cdot))$ is the gradient of $\varphi$ with respect to $z$.

Similarly, the mapping $\mathbb G \ni (\tau,z,w(\cdot)) \mapsto \phi = (\phi_1, \ldots, \phi_n) \in \mathbb R^n$ is called ci-differentiable at a point $(\tau,z,w(\cdot)) \in \mathbb G$, $\tau < \vartheta$, if the functionals $\phi_i\colon \mathbb G \mapsto \mathbb R$, $i=\overline{1,n}$ are ci-differentiable at this point.

Let us fix the functions $g \colon [0,\vartheta] \times \mathbb R^n \mapsto \mathbb R^n$ and $H\colon [0,\vartheta] \times \mathbb R^n \times \mathbb R^n \times \mathbb R^n \mapsto \mathbb R$ and the mapping $\sigma \colon \mathbb R^n \times \mathrm{PLip} \mapsto \mathbb R$ satisfying the the following conditions:
\begin{itemize}

\item[$(g)$] The function $g$ is continuously differentiable.

\item[$(H_1)$] The function $H$ is continuous.
	
\item[$(H_2)$] There exists a constant $c_H > 0$ such that
\begin{equation*}
\big|H(\tau,z,r,s) - H(\tau,z,r,s')\big| \leq c_H \big(1 + \|z\| + \|r\|\big) \|s - s'\|
\end{equation*}
for any $\tau \in [0,\vartheta]$, $z,r,s,s' \in \mathbb R^n$.

\item[$(H_3)$] For every $\alpha > 0$, there exists $\lambda_H = \lambda_H(\alpha) > 0$ such that
\begin{equation*}
\big|H(\tau,z,r,s) - H(\tau,z',r',s)\big|\leq \lambda_H \big(\|z - z'\| + \|r - r'\|\big) \big(1 + \|s\|\big)
\end{equation*}
for any $\tau \in [0,\vartheta]$, $z,r,z',r',s \in \mathbb R^n$: $\max\{\|z\|,\|r\|,\|z'\|,\|r'\|\} \leq \alpha$.

\item[$(\sigma)$]
For every $\alpha > 0$, there exists $\lambda_\sigma = \lambda_\sigma(\alpha) > 0$ such that
\begin{equation*}
\big|\sigma(z,w(\cdot)) - \sigma(z',w'(\cdot))\big| \leq \lambda_\sigma \big(\|z - z'\| + \|w(\cdot) - w'(\cdot)\|_1\big)
\end{equation*}
for any $(z,w(\cdot)),(z',w'(\cdot)) \in P(\alpha)$, where
\begin{equation}\label{def:P}
P(\alpha) = \big\{(z,w(\cdot)) \in \mathbb R^n \times \mathrm{PLip} \colon \|z\| \leq \alpha,\, \|w(\cdot)\|_\infty \leq \alpha\big\}.
\end{equation}
\end{itemize}

Consider the mapping $g_*(\tau,z,w(\cdot)) = g(\tau,w(-h))$, $(\tau,z,w(\cdot)) \in \mathbb G$. Due to condition $(g)$, $g_*$ is ci-differentiable on $\mathbb G_*$ (see (\ref{def:G_G_s})) and
\begin{equation}\label{def:dg_s}
\partial^{ci}_{\tau,w} g_*(\tau,z,w(\cdot)) = G(\tau,w(-h),\mathrm{d}^+ w(-h) / \mathrm{d} \xi),\quad \nabla_z g_*(\tau,z,w(\cdot)) = 0
\end{equation}
for any $(\tau,z,w(\cdot)) \in \mathbb G_*$, where $\mathrm{d}^+ w(-h) / \mathrm{d} \xi$ is the right derivative of the function $w(\xi)$, $\xi \in [-h,0)$ at the point $\xi=-h$ and $G(\tau,x,y) = \partial g(\tau,x) / \partial \tau + \nabla_x g(\tau,x) y$. Since the mapping $g_*$ is determined by the function $g$ and does not depend on $z$, for brevity, we denote
\begin{equation}\label{def:derivative_g}
\partial^{ci}_{\tau,w} g(\tau,w(\cdot)) = \partial^{ci}_{\tau,w} g_*(\tau,z,w(\cdot)).
\end{equation}

For the functional $\varphi \colon \mathbb G \mapsto \mathbb R$, let us consider the Cauchy problem for the HJ equation
\begin{equation}\label{Hamilton-Jacobi_equation}
\begin{array}{rcl}
\partial^{ci}_{\tau,w} \varphi(\tau,z,w(\cdot)) & + & \langle\partial^{ci}_{\tau,w} g(\tau,w(\cdot)), \nabla_z \varphi(\tau,z,w(\cdot)) \rangle \\[0.3cm]
& + & H(\tau,z,w(-h),\nabla_z \varphi(\tau,z,w(\cdot))) = 0,
\end{array}
\ (\tau,z,w(\cdot)) \in \mathbb G_*,
\end{equation}
and the terminal condition
\begin{equation}\label{terminal_condition}
\varphi(\vartheta,z,w(\cdot)) =\sigma(z,w(\cdot)),\quad (z,w(\cdot)) \in \mathbb R^n \times \mathrm{PLip}.
\end{equation}

\begin{remark}
Such Cauchy problems arise in investigations of optimal control problems and differential games for neutral-type systems in Hale's from (see, e.g., \cite{Gomoyunov_Plaksin_2019}). In contrast to Cauchy problems corresponding to time-delay systems \cite{Lukoyanov_2000,Lukoyanov_2003,Lukoyanov_2010a,Lukoyanov_2010b}, the new term $\langle \partial^{ci}_{\tau,w} g(\tau,w(\cdot)), \nabla_z \varphi(\tau,z,w(\cdot)) \rangle$ appears. Since the functional $\partial^{ci}_{\tau,w} g(\tau,w(\cdot))$ depends on the derivative $\mathrm{d}^+ w(-h) / \mathrm{d} \xi$ (see (\ref{def:dg_s}), (\ref{def:derivative_g})), it is not defined at all points of $\mathbb G$ and discontinuous with respect to the uniform norm. Thus, we can not apply results from \cite{Lukoyanov_2000,Lukoyanov_2003,Lukoyanov_2010a,Lukoyanov_2010b} to Cauchy problem (\ref{Hamilton-Jacobi_equation}), (\ref{terminal_condition}). Moreover, we need to consider the HJ equation only on a set on which $\partial^{ci}_{\tau,w} g(\tau,w(\cdot))$ is defined. We choice $\mathbb G_*$ as such a set. Due to condition $(g)$, the value $\partial^{ci}_{\tau,w} g(\tau,w(\cdot))$ is defined on $\mathbb G_*$.
\end{remark}

We search a solution of this problem among the functionals $\varphi\colon \mathbb G \mapsto \mathbb R$ satisfying the following conditions:
\begin{itemize}
\item[$(\varphi_1)$] For each $(\tau,w(\cdot)) \in [0,\vartheta] \times \mathrm{Lip}$, the function $\overline{\varphi}(t) = \varphi(t,w(-0),w(\cdot))$, $t \in [\tau,\vartheta]$ is continuous.

\item[$(\varphi_2)$] For every $\alpha > 0$, there exists $\lambda_\varphi = \lambda_\varphi(\alpha) > 0$ such that
\begin{equation}\label{Phi_lip}
|\varphi(\tau,z,w(\cdot)) - \varphi(\tau,z',w'(\cdot))| \leq \lambda_\varphi \upsilon(\tau,z-z',w(\cdot) - w'(\cdot))
\end{equation}
for any $\tau \in [0,\vartheta]$ and $(z,w(\cdot)), (z',w'(\cdot)) \in P(\alpha)$, where
\begin{equation}\label{def:upsilon}
\upsilon(\tau,z,w(\cdot)) = \|z\| + \|w(\cdot)\|_1 + \|w(-h)\| + \|w(i h - \tau)\|
\end{equation}
in which $i \in \overline{-1,I-1}$ is such that $\tau \in (i h, (i+1) h]$.
\end{itemize}

\begin{remark}\label{rem:conditions}
Note that if $\varphi \colon \mathbb G \mapsto \mathbb R$ satisfies conditions $(\varphi_1)$, $(\varphi_2)$, then the functional $\hat{\varphi}(\tau,w(\cdot)) = \varphi(\tau,w(-0),w(\cdot))$, $(\tau,w(\cdot))\in [0,\vartheta] \times \mathrm{Lip}$ is continuous with respect to uniform norm. Thus, these conditions are consistent with prior works \cite{Gomoyunov_Plaksin_2019,Lukoyanov_Plaksin_2020,Plaksin_2021b} in which solutions of HJ equations for neutral-type systems on the space of Lipschitz continuous functions were considered in the class of continuous functionals. Moreover, these
conditions are an analog of the conditions suggested in \cite{Plaksin_2020,Plaksin_2021}, devoted to viscosity solutions of HJ equations for time-delay systems, with additional terms $\|w(-h)\|$ and $\|w(i h - \tau)\|$.
\end{remark}

\subsection{Viscosity solution}

Denote
\begin{equation*}
\Lambda_0(\tau,z,w(\cdot)) = \big\{x(\cdot) \in \Lambda(\tau,z,w(\cdot)) \colon x(t) = z,\ t \in [\tau,\vartheta]\big\},\quad (\tau,z,w(\cdot)) \in \mathbb G.
\end{equation*}
\begin{remark}
Similar to \cite{Plaksin_2021}, note that if a functional $\varphi \colon \mathbb G \mapsto \mathbb R$
satisfies conditions $(\varphi_1)$, $(\varphi_2)$ and is ci-differentiable at $(\tau,z,w(\cdot)) \in \mathbb G_*$, then the function
\begin{equation}\label{tilde_phi}
\tilde{\varphi}(t,x) = \varphi(t,x,\kappa_t(\cdot)),\quad (t,x) \in \mathbb [\tau,\vartheta] \times \mathbb R^n,\quad \kappa(\cdot) \in \Lambda_0(\tau,z,w(\cdot)),
\end{equation}
has a right partial derivative $\partial^+ \tilde{\varphi}(\tau,z) / \partial \tau$ and a gradient $\nabla_z \tilde{\varphi}(\tau,z)$ at the point $(\tau,z)$, and satisfies the following HJ equation:
\begin{equation*}
\partial^+ \tilde{\varphi}(\tau,z) / \partial \tau
+ \langle\partial^{ci}_{\tau,w} g(\tau,w(\cdot)), \nabla_z \tilde{\varphi}(\tau,z) \rangle
 + H(\tau,z,w(-h),\nabla_z \tilde{\varphi}(\tau,z)) = 0.
\end{equation*}
Thus, we might say that HJ equation with a ci-derivative (\ref{Hamilton-Jacobi_equation}) is locally this HJ equation with partial derivatives.
\end{remark}

Then, based on the classical definition of viscosity solutions \cite{Crandall_Lions_1983}, it leads us in a natural way to the following definition.

\begin{definition}\label{def:viscosity_solution}
A functional $\varphi \colon \mathbb G \mapsto \mathbb R$ is called a viscosity solution of problem (\ref{Hamilton-Jacobi_equation}), (\ref{terminal_condition}) if $\varphi$ satisfies conditions $(\varphi_1)$, $(\varphi_2)$, (\ref{terminal_condition}) and conditions
\begin{subequations}
\begin{gather}
\left\{
\begin{array}{ll}
\text{for every}\ (\tau,z,w(\cdot)) \in \mathbb G_*,\ \psi \in \mathrm{C}^1(\mathbb R \times \mathbb R^n,\mathbb R), \text{ and } \delta > 0,\\[0.2cm]
\text{if}\ \varphi(\tau,z,w(\cdot)) - \psi(\tau,z) \leq \varphi(t,x,\kappa_t(\cdot)) - \psi(t,x)\\[0.2cm]
\text{for any}\ (t,x) \in O_\delta^+(\tau,z),\ \kappa(\cdot) \in \Lambda_0(\tau,z,w(\cdot)), \\[0.2cm]
\text{then}\ \partial \psi(\tau,z) / \partial \tau + \langle \partial^{ci}_{\tau,w} g(\tau,w(\cdot)), \nabla_z \psi(\tau,z) \rangle \\[0.2cm]
\hspace{4cm} + H(\tau,z,w(-h),\nabla_z \psi(\tau,z)) \leq 0,
\end{array} \right.\label{sub_viscosity_solution}\\[0.2cm]
\left\{
\begin{array}{ll}
\text{for every}\ (\tau,z,w(\cdot)) \in \mathbb G_*,\ \psi \in \mathrm{C}^1(\mathbb R \times \mathbb R^n,\mathbb R), \text{ and } \delta > 0,\\[0.2cm]
\text{if}\ \varphi(\tau,z,w(\cdot)) - \psi(\tau,z) \geq \varphi(t,x,\kappa_t(\cdot)) - \psi(t,x)\\[0.2cm]
\text{for any}\ (t,x) \in O_\delta^+(\tau,z),\ \kappa(\cdot) \in \Lambda_0(\tau,z,w(\cdot)), \\[0.2cm]
\text{then}\ \partial \psi(\tau,z) / \partial \tau + \langle \partial^{ci}_{\tau,w} g(\tau,w(\cdot)), \nabla_z \psi(\tau,z) \rangle \\[0.2cm]
\hspace{4cm} + H(\tau,z,w(-h),\nabla_z \psi(\tau,z)) \geq 0.
\end{array}\label{super_viscosity_solution} \right.
\end{gather}
\end{subequations}
Here $O^+_\delta(\tau,z) = \{(t,x) \in [\tau,\tau+\delta] \times \mathbb R^n \colon \|x - z\| \leq \delta\}$, and $\mathrm{C}^1(\mathbb R \times \mathbb R^n, \mathbb R)$ is the space of continuously differentiable functions $\psi \colon \mathbb R \times \mathbb R^n \mapsto \mathbb R$.
\end{definition}
The main result of the paper is the following theorem.
\begin{theorem}\label{teo:viscosity_solution}
There exists a unique viscosity solution $\varphi$ of problem {\rm (\ref{Hamilton-Jacobi_equation}), (\ref{terminal_condition})}.
\end{theorem}
In order to prove the theorem, we consider the minimax approach to the definition of generalized solutions of problem (\ref{Hamilton-Jacobi_equation}), (\ref{terminal_condition}) (see \cite{Plaksin_2021b}).

\subsection{Minimax solution}

Taking the constant $c_H > 0$ from $(H_2)$, denote
\begin{equation}\label{def:F}
F^\eta(x,y) = \big\{l \in \mathbb R^n \colon \|l\| \leq c_H (1 + \|x\| + \|y\|) + \eta \big\},\ \ x,y \in \mathbb R^n,\ \ \eta \in [0,1].
\end{equation}
Let $(\tau,z,w(\cdot)) \in \mathbb G$. Denote by $X^\eta(\tau,z,w(\cdot))$ the set of the functions $x(\cdot) \in \Lambda(\tau,z,w(\cdot))$ such that the function $y(t) = x(t) - g(t,x(t-h))$, $t \in [\tau,\vartheta]$ is Lipschitz continuous and the neutral-type differential inclusion
\begin{equation}\label{def:inclusion}
\frac{\mathrm{d}}{\mathrm{d} t}\Big(x(t) - g(t,x(t-h))\Big) \in F^\eta(x(t),x(t-h))\text{ for a.e. } t \in [\tau,\vartheta],
\end{equation}
holds. Note that the set $X^\eta(t,z,w(\cdot))$ is not empty. At least, the function $x(\cdot) \in \Lambda(\tau,z,w(\cdot))$ satisfying $x(t) = g(t,x(t-h))$, $t \in (\tau,\vartheta]$, belongs to $X^\eta(t,z,w(\cdot))$.

The following inequalities for functionals $\varphi \colon \mathbb G \mapsto \mathbb R$ are key to define minimax solutions under consideration below:
\begin{subequations}
\begin{align}
\inf\limits_{x(\cdot) \in X^\eta(\tau,z,w(\cdot))} \big(\varphi(t,x(t),x_t(\cdot)) +
\omega(\tau,t,x(\cdot),s)\big) \leq \varphi(\tau,z,w(\cdot)), \label{def:upper_minmax_solution} \\[0.0cm]
\sup\limits_{x(\cdot) \in X^\eta(\tau,z,w(\cdot))} \big(\varphi(t,x(t),x_t(\cdot)) + \omega(\tau,t,x(\cdot),s)\big) \geq \varphi(\tau,z,w(\cdot)), \label{def:lower_minmax_solution}
\end{align}
\end{subequations}
where $(\tau,z,w(\cdot)) \in \mathbb G$, $\tau < \vartheta$, $t \in (\tau,\vartheta]$, $s \in \mathbb R^n$, $\eta \in [0,1]$, and
\begin{equation}\label{def:omega}
\begin{array}{rcl}
\omega(\tau,t,x(\cdot),s) &=& \displaystyle\int_\tau^t H(\xi,x(\xi),x(\xi-h),s) \mathrm{d} \xi \\[0.4cm]
 &-& \big\langle \Big(x(t) - g(t,x(t-h))\Big) - \Big(x(\tau) - g(\tau,x(\tau-h))\Big), s \big\rangle.
\end{array}
\end{equation}

\begin{definition}
A functional $\varphi \colon \mathbb G \mapsto \mathbb R$ is called a minimax solution of problem {\rm (\ref{Hamilton-Jacobi_equation}), (\ref{terminal_condition})}, if $\varphi$ satisfies conditions $(\varphi_1)$, $(\varphi_2)$, {\rm (\ref{terminal_condition})} and the inequalities (\ref{def:upper_minmax_solution}), (\ref{def:lower_minmax_solution}) for any
$(\tau,z,w(\cdot)) \in \mathbb G$, $\tau < \vartheta$, $t \in (\tau,\vartheta]$, $s \in \mathbb R^n$, and $\eta = 0$.
\end{definition}

This definition of minimax solution seems to be natural for considered problem (\ref{Hamilton-Jacobi_equation}), (\ref{terminal_condition}). Nevertheless, recently, the theory of minimax solution of problems similar to (\ref{Hamilton-Jacobi_equation}), (\ref{terminal_condition}) was developed on the space of Lipschitz continuous functions  (see \cite{Plaksin_2021b}). To apply these results, it is convenient to give the following auxiliary definition of a minimax solution.

\begin{definition}
A functional $\varphi \colon \mathbb G \mapsto \mathbb R$ is called a $\mathrm{Lip}$-minimax solution of problem {\rm (\ref{Hamilton-Jacobi_equation}), (\ref{terminal_condition})}, if $\varphi$ satisfies conditions $(\varphi_1)$, $(\varphi_2)$ and, taking $z = w(-0)$, satisfies condition {\rm (\ref{terminal_condition})} for any $w(\cdot) \in \mathrm{Lip}$ and inequalities (\ref{def:upper_minmax_solution}), (\ref{def:lower_minmax_solution}) for any $(\tau,w(\cdot)) \in [0,\vartheta) \times \mathrm{Lip}$, $t \in (\tau,\vartheta]$, $s \in \mathbb R^n$, and $\eta =0$.
\end{definition}
The first step to prove Theorem \ref{teo:viscosity_solution} is the following theorem which follows from Lemmas \ref{lem:ex_un_minimax_sol}, \ref{lem:varphi_extension} taking Remark \ref{rem:conditions} into account.
\begin{theorem}\label{teo:Lip_minimax_solution}
There exists a unique $\mathrm{Lip}$-minimax solution $\varphi$ of problem {\rm(\ref{Hamilton-Jacobi_equation}),~(\ref{terminal_condition})}.
\end{theorem}

Then, we introduce one more auxiliary definition of a minimax solution.
\begin{definition}\label{def:C1_minimax_solution}
A functional $\varphi \colon \mathbb G \mapsto \mathbb R$ is called a $\mathrm{C^1}$-minimax solution of problem {\rm (\ref{Hamilton-Jacobi_equation}), (\ref{terminal_condition})}, if $\varphi$ satisfies conditions $(\varphi_1)$, $(\varphi_2)$,
and, taking $z = w(-0)$, satisfies condition {\rm (\ref{terminal_condition})} for any $w(\cdot) \in \mathrm{C}^1$ and inequalities (\ref{def:upper_minmax_solution}), (\ref{def:lower_minmax_solution}) for any $i \in \overline{0,I-1}$, $(\tau,w(\cdot)) \in [i h, (i + 1) h) \times \mathrm{C}^1$, $t \in (\tau,(i+1)h]$, $s \in \mathbb R^n$, and $\eta \in (0,1]$.
\end{definition}
The following theorem establishes the equivalence of these three definitions.

\begin{theorem}\label{teo:minimax_solutions}
The following statements are equivalent:
\begin{description}
\item[(a)] The functional $\varphi \colon \mathbb G \mapsto \mathbb R$ is a minimax solution of problem {\rm (\ref{Hamilton-Jacobi_equation}), (\ref{terminal_condition})}.
\vspace{0.1cm}

\item[(b)] The functional $\varphi \colon \mathbb G \mapsto \mathbb R$ is a $\mathrm{Lip}$-minimax solution of problem {\rm (\ref{Hamilton-Jacobi_equation}), (\ref{terminal_condition})}.
\vspace{0.1cm}

\item[(c)] The functional $\varphi \colon \mathbb G \mapsto \mathbb R$ is a $\mathrm{C^1}$-minimax solution of problem {\rm (\ref{Hamilton-Jacobi_equation}), (\ref{terminal_condition})}.
\end{description}
\end{theorem}
The implications $(a) \Rightarrow (b) \Rightarrow (c)$ are valid due to inclusions (\ref{space_inclusions}) and inclusion $X^0(\tau,z,w(\cdot)) \subset X^\eta(\tau,z,w(\cdot))$ for any $\eta \in (0,1]$. The implication $(c) \Rightarrow (a)$ is proved in Lemma \ref{lem:C1_and_minimax_solutions}.

Finally, in order to prove Theorem \ref{teo:viscosity_solution}, we follow scheme from \cite{Clarke_Ledyaev_1994,Subbotin_1993,Subbotin_1995} (see also \cite{Plaksin_2020}). For that, we define the corresponding notions of lower and upper right directional derivatives as well as the notions of subdifferential and superdifferential of a functional $\varphi \colon \mathbb G \mapsto \mathbb R$.

By analogy with \cite{Plaksin_2021} (see also \cite{Lukoyanov_2010a}), lower and upper right directional derivatives of a functional $\varphi \colon \mathbb G \mapsto \mathbb R$ along $(l_0, l) \in \mathbb [0,+\infty) \times \mathbb R^n$ at $(\tau,z,w(\cdot)) \in \mathbb G$, $\tau < \vartheta$ are defined by
\begin{subequations}\label{directional_derivatives}
\begin{align}
\partial^-_{(l_0,l)} \varphi(\tau,z,w(\cdot)) &
= \liminf\limits_{\delta \downarrow 0} \displaystyle\frac{\varphi(\tau + l_0 \delta, z + l \delta,\kappa_{\tau + l_0 \delta}(\cdot)) - \varphi(\tau,z,w(\cdot))}{\delta},\label{lower_directional_derivatives}\\[0.0cm]
\partial^+_{(l_0,l)} \varphi(\tau,z,w(\cdot)) &
= \limsup\limits_{\delta \downarrow 0} \displaystyle\frac{\varphi(\tau + l_0 \delta, z + l \delta,\kappa_{\tau + l_0 \delta}(\cdot)) - \varphi(\tau,z,w(\cdot))}{\delta},\label{ipper_directional_derivatives}
\end{align}
\end{subequations}
where $\kappa(\cdot) \in \Lambda_0(\tau,z,w(\cdot))$.

The subdifferential of a functional $\varphi \colon \mathbb G \mapsto \mathbb R$ at a point $(\tau,z,w(\cdot)) \in \mathbb G$, $\tau < \vartheta$ is a set, denoted by $D^-(\tau,z,w(\cdot))$, of pairs $(p_0,p) \in \mathbb R \times \mathbb R^n$ such that
\begin{equation}\label{subdifferential}
\lim\limits_{\delta \to 0} \inf\limits_{(t,x) \in O^+_\delta(\tau,z)}
\!\!\frac{\varphi(t,x,\kappa_t(\cdot)) - \varphi(\tau,z,w(\cdot)) - (t - \tau) p_0 - \langle x - z, p \rangle}{|t - \tau| + \|x - z\|} \geq 0.
\end{equation}
The superdifferential of a functional $\varphi \colon \mathbb G \mapsto \mathbb R$ at a point $(\tau,z,w(\cdot)) \in \mathbb G$, $\tau < \vartheta$ is a set, denoted by $D^+(\tau,z,w(\cdot))$, of pairs $(q_0,q) \in \mathbb R \times \mathbb R^n$ such that
\begin{equation}\label{superdifferential}
\lim\limits_{\delta \to 0} \sup\limits_{(t,x) \in O^+_\delta(\tau,z)}
\!\!\frac{\varphi(t,x,\kappa_t(\cdot)) - \varphi(\tau,z,w(\cdot)) - (t - \tau) q_0 - \langle x - z, q \rangle}{|t - \tau| + \|x - z\|} \leq 0.
\end{equation}
The theorem below together with Theorem \ref{teo:Lip_minimax_solution} complete the proof of Theorem~\ref{teo:viscosity_solution}.

\begin{theorem}\label{teo:equivalent_solutions}
The following statements are equivalent:
\begin{description}
\item[(a)] The functional $\varphi \colon \mathbb G \mapsto \mathbb R$ is a minimax solution of problem {\rm (\ref{Hamilton-Jacobi_equation}), (\ref{terminal_condition})}.
\vspace{0.1cm}

\item[(b)] The functional $\varphi \colon \mathbb G \mapsto \mathbb R$ is a $\mathrm{C^1}$-minimax solution of problem {\rm (\ref{Hamilton-Jacobi_equation}), (\ref{terminal_condition})}.
\vspace{0.1cm}

\item[(c)] The functional $\varphi \colon \mathbb G \mapsto \mathbb R$ satisfies conditions $(\varphi_1)$, $(\varphi_2)$, {\rm (\ref{terminal_condition})} and, for every
$(\tau,z,w(\cdot)) \in \mathbb G_*$ and $s \in \mathbb R^n$, the following inequalities hold:
\begin{subequations}
\begin{gather}
\begin{array}{l}
\inf\limits_{l \in F^0(z,w(-h)) + \partial^{ci}_{\tau,w} g(\tau,w(\cdot))}
\Big(\partial^-_{1,l} \varphi(\tau,z,w(\cdot)) + \langle \partial^{ci}_{\tau,w} g(\tau,w(\cdot)) ,s \rangle\\[0.0cm]
\hspace{5cm} + H(\tau,z,w(-h),s) - \langle l,s \rangle\Big) \leq 0,\label{lower_directional_derivative_inequality} \\[0.1cm]
\end{array}\\[0.1cm]
\begin{array}{l}
\sup\limits_{l \in F^0(z,w(-h)) + \partial^{ci}_{\tau,w} g(\tau,w(\cdot))} \Big(\partial^+_{1,l} \varphi(\tau,z,w(\cdot)) + \langle \partial^{ci}_{\tau,w} g(\tau,w(\cdot)) ,s \rangle\\[0.0cm]
\hspace{5cm} + H(\tau,z,w(-h),s) - \langle l,s \rangle\Big) \geq 0.\label{upper_directional_derivative_inequality}
\end{array}
\end{gather}
\end{subequations}
\vspace{0.1cm}

\item[(d)] The functional $\varphi \colon \mathbb G \mapsto \mathbb R$ satisfies conditions $(\varphi_1)$, $(\varphi_2)$, {\rm (\ref{terminal_condition})} and, for every $(\tau,z,w(\cdot)) \in \mathbb G_*$, $(p_0,p) \in D^-\varphi(\tau,z,w(\cdot))$, and $(q_0,q) \in D^+\varphi(\tau,z,w(\cdot))$ the following inequalities hold:
\begin{subequations}
\begin{gather}
p_0 + \langle \partial^{ci}_{\tau,w} g(\tau,w(\cdot)), p \rangle + H(\tau,z,w(-h),p) \leq 0,\label{subdifferential_viscosity_inequality}\\[0.1cm]
q_0 + \langle \partial^{ci}_{\tau,w} g(\tau,w(\cdot)), p \rangle + H(\tau,z,w(-h),q) \geq 0.\label{superdifferential_viscosity_inequality}
\end{gather}
\end{subequations}

\item[(e)] The functional $\varphi \colon \mathbb G \mapsto \mathbb R$ is a viscosity solution of problem {\rm (\ref{Hamilton-Jacobi_equation}), (\ref{terminal_condition})}.
\end{description}
\end{theorem}
The implications $(a) \Rightarrow (e) \Rightarrow (d) \Rightarrow (c) \Rightarrow (b)  \Rightarrow (a)$ follow from Lemmas \ref{minimax_and_viscosity_solutions}, \ref{lem:viscosity_and_D_solutions}, \ref{lem:D_and_dini_solutions}, \ref{lem:dini_and_C1_solutions}, and Theorem \ref{teo:minimax_solutions}, respectively.

\begin{remark}
Besides the proof of Theorem \ref{teo:viscosity_solution}, Theorem \ref{teo:equivalent_solutions} establishs other equivalent definitions of generalized solution of (\ref{Hamilton-Jacobi_equation}), (\ref{terminal_condition}) which have analogues in the classical theory \cite{Subbotin_1995}. Namely, $(c)$ is similar to the infinitesimal variant of the definitions of minimax solutions (see \cite{Subbotin_1980}) or Dini solutions according to another terminology (see \cite{Clarke_Ledyaev_1994}). $(d)$ is an analog of the equivalent notion of the viscosity solutions considered in \cite{Crandall_Evans_Lions_1984}.
\end{remark}

\subsection{Consistency}

\begin{remark}
Note that, directly from definitions of ci-differentiability (see (\ref{def:ci-differentiability})) and sub- and superdifferentials (see (\ref{subdifferential}) and (\ref{superdifferential})), if a functional $\varphi \colon \mathbb G \mapsto \mathbb R$ satisfies conditions $(\varphi_1)$, $(\varphi_2)$ and is ci-differentiable $(\tau,z,w(\cdot)) \in \mathbb G_*$, then
\begin{equation*}
\begin{array}{rcl}
D^-\varphi(\tau,z,w(\cdot)) &=& \big\{(p_0,p) \colon p_0 \leq \partial^{ci}_{\tau,w}\varphi(\tau,z,w(\cdot)),\, p = \nabla_z\varphi(\tau,z,w(\cdot))\big\},\\[0.2cm]
D^+\varphi(\tau,z,w(\cdot)) &=& \big\{(q_0,q) \colon q_0 \geq \partial^{ci}_{\tau,w}\varphi(\tau,z,w(\cdot)),\, q = \nabla_z\varphi(\tau,z,w(\cdot))\big\}.
\end{array}
\end{equation*}
\end{remark}
Hence, the following statement about consistency of the viscosity solution definition (see Definition \ref{def:viscosity_solution}) and problem (\ref{Hamilton-Jacobi_equation}), (\ref{terminal_condition}) can be obtained from Theorem~\ref{teo:equivalent_solutions}.
\begin{theorem}\label{teo:classical_and_viscosity_solutions}
a) Let a functional $\varphi \colon \mathbb G \mapsto \mathbb R$ be the viscosity solution of problem {\rm (\ref{Hamilton-Jacobi_equation}), (\ref{terminal_condition})}. If $\varphi$ is ci-differentiable at a point $(\tau,z,w(\cdot)) \in \mathbb G_*$, then it satisfies HJ equation {\rm (\ref{Hamilton-Jacobi_equation})} at this point. b) Let a functional $\varphi \colon \mathbb G \mapsto \mathbb R$ be ci-differentiable at every $(\tau,z,w(\cdot)) \in \mathbb G_*$, satisfy HJ equation {\rm (\ref{Hamilton-Jacobi_equation})} on $\mathbb G_*$ and satisfy conditions $(\varphi_1)$, $(\varphi_1)$, and {\rm (\ref{terminal_condition})}. Then $\varphi$ is the viscosity solution of problem {\rm (\ref{Hamilton-Jacobi_equation}),~(\ref{terminal_condition})}.
\end{theorem}

\subsection{Application for optimal control problems}

In spite of the fact that prior works (see, e.g., \cite{Gomoyunov_Plaksin_2019,Lukoyanov_Plaksin_2020b}) consider dynamical optimization problems for neutral-time systems on the space of Lipschitz continuous functions, we can also apply the results of this paper to get characterizations of value functionals in such problems.

For example, consider the following optimal control problem: for each $(\tau,w(\cdot)) \in [0,\vartheta] \times \mathrm{Lip}$, it is required to minimize the Bolza cost functional
\begin{equation}\label{J}
J(\tau,w(\cdot),u(\cdot)) = \sigma(x(\vartheta),x_\vartheta(\cdot)) + \int_\tau^\vartheta f^0(t,x(t),x(t-h),u(t)) \mathrm{d} t,
\end{equation}
over all measurable functions $u(\cdot)\colon [\tau,\vartheta] \mapsto \mathbb U$, where $x(\cdot) \in \Lambda(\tau,w(-0),w(\cdot))$ is a Lipschitz continuous function satisfying the neutral-type equation
\begin{equation}\label{def:neutral_type_equation}
\frac{\mathrm{d}}{\mathrm{d} t}\Big(x(t) - g(t,x(t-h))\Big) = f(t,x(t),x(t-h),u(t))\text{ for a.e. } t \in [\tau,\vartheta].
\end{equation}
Here $\mathbb U \subset \mathbb R^m$ is the compact set; the function $g$ and the functional $\sigma$ satisfy conditions $(g)$ and $(\sigma)$; the functions $f \colon [0,\vartheta] \times \mathbb R^n \times \mathbb R^n \times \mathbb U \mapsto \mathbb R^n$ and $f^0 \colon [0,\vartheta] \times \mathbb R^n \times \mathbb R^n \times \mathbb U \mapsto \mathbb R$ satisfy the following conditions:
\begin{itemize}
\item[$(f_1)$] The functions $f$ and $f^0$ are continuous.
	
\item[$(f_2)$] There exists a constant $c_f > 0$ such that
\begin{equation*}
\big\|f(t,z,r,u)\big\| + \big|f^0(t,z,r,u)\big| \leq c_f\big(1 + \|z\| + \|r\|\big)
\end{equation*}
for any $t \in [0,\vartheta]$, $z,r \in \mathbb R^n$, and $u \in \mathbb U$.

\item[$(f_3)$] For every $\alpha > 0$, there exists $\lambda_f = \lambda_f(\alpha) > 0$ such that
\begin{equation*}
\begin{array}{c}
\big\|f(t,z,r,u) - f(t,z',r',u)\big\| + \big|f^0(t,z,r,u) - f^0(t,z',r',u)\big|\\[0.2cm]
\leq \lambda_f\big(\|z - z'\| + \|r - r'\|\big)
\end{array}
\end{equation*}
for any $t \in [0,\vartheta]$, $z,r,z',r' \in \mathbb R^n$: $\max\{\|z\|,\|r\|,\|z'\|,\|r'\|\} \leq \alpha$, and $u \in \mathbb U$.
\end{itemize}
Note that, under these conditions, the function
\begin{equation}\label{def:hamiltonian}
H(t,z,r,s) = \min\limits_{u \in \mathbb U} \big(\langle f(t,z,r,u), s \rangle + f^0(t,z,r,u)\big),\ t \in [0,\vartheta],\ z,r,s \in \mathbb R^n,
\end{equation}
satisfies conditions $(H_1)$--$(H_3)$. The value functional of this problem is
\begin{equation}\label{def:value_functional}
\hat{\varphi}(\tau,w(\cdot)) = \inf\limits_{u(\cdot)} J(\tau,w(\cdot),u(\cdot)).
\end{equation}

Applying Lemma \ref{lem:value_functional} to show that $\hat{\varphi}$ is the functional from Lemma \ref{lem:ex_un_minimax_sol} and next, using Lemma \ref{lem:varphi_extension} and Theorem \ref{teo:equivalent_solutions}, taking Remark \ref{rem:conditions} into account, we can obtain the following result.
\begin{theorem}\label{teo:application}
There exists a unique functional $\varphi \colon \mathbb G \mapsto \mathbb R$ satisfying conditions $(\varphi_1)$, $(\varphi_2)$ and the equality
\begin{equation*}
\varphi(\tau,w(-0),w(\cdot)) = \hat{\varphi}(\tau,w(\cdot)),\quad
(\tau,w(\cdot)) \in [0,\vartheta] \times \mathrm{Lip}.
\end{equation*}
Such a functional $\varphi$ is a unique viscosity solution of problem {\rm (\ref{Hamilton-Jacobi_equation}), (\ref{terminal_condition})}.
\end{theorem}

\section{Proofs}

\subsection{Auxiliary statements}

\begin{lemma}\label{lem:alpha_x_lambda_x}
For every $\alpha > 0$, there exists $\alpha_X = \alpha_X(\alpha) > 0$, $\alpha_X^* = \alpha_X^*(\alpha) > 0$, and $\lambda_X^* = \lambda_X^*(\alpha) > 0$ such that, for each $\tau \in [0,\vartheta]$ and $(z,w(\cdot)) \in P(\alpha)$ {\em(}see {\em(\ref{def:P}))}, the functions  $x(\cdot) \in X^1(\tau,z,w(\cdot))$ and $y(t) = x(t) - g(t,x(t-h))$, $t \in [\tau,\vartheta]$, satisfy the relations
\begin{equation*}
(x(t),x_t(\cdot)) \in P(\alpha_X),\ \ \|y(t)\|\leq \alpha_X^*,\ \ \|y(t) - y(t')\| \leq \lambda^*_X |t - t'|,\ \
t,t' \in [\tau,\vartheta].
\end{equation*}
\end{lemma}
\begin{proof}
The existence of $\alpha_X$ satisfying the first relation can be proved in the same way as in Lemma 6.1 from \cite{Lukoyanov_Plaksin_2020}. Due to condition $(g)$, there exists $\alpha_g > 0$ such that $\|g(\tau,w(-h))\| \leq \alpha_g$ for any $\tau \in [0,\vartheta]$ and $(z,w(\cdot)) \in P(\alpha)$. Then, from (\ref{def:F}), (\ref{def:inclusion}), the second relation is obtained by
\begin{equation*}
\begin{array}{c}
\|y(t)\| \leq \|z\| + \|g(\tau,w(-h))\| + \displaystyle\int_\tau^t \Big(c_H \big(1 + \|x(\xi)\| + \|x(\xi-h)\|\big) + 1\Big) \mathrm{d} \xi \\[0.4cm]
\leq \alpha + \alpha_g + (c_H (1 + 2 \alpha_X) + 1) \vartheta:= \alpha_X^*,\quad t \in [\tau,\vartheta],
\end{array}
\end{equation*}
and, taking $\lambda_X^* = c_H (1 + 2 \alpha_X)$, the third relation is derived by
\begin{equation*}
\|y(t') - y(t)\| \leq c_H \int_{t}^{t'} \Big(c_H \big(1 + \|x(\xi)\| + \|x(\xi-h)\|\big) + 1\Big) \mathrm{d} \xi \leq \lambda_X^* (t' - t),
\end{equation*}
where, without loss of generality, we assume $\tau \leq t \leq t' \leq \vartheta$.
\end{proof}

\begin{lemma}\label{lem:g_s}
Let the function $g \colon [0,\vartheta] \times \mathbb R^n \mapsto \mathbb R^n$ satisfy $(g)$. Then, for every $\alpha > 0$ there exists $\lambda_g = \lambda_g(\alpha) > 0$ such that
\begin{equation*}
\big|g(t,x) - g(t',x')\big| \leq \lambda_g \big(|t - t'| + \|x - x'\|\big)
\end{equation*}
for any $(t,x), (t',x') \in [0,\vartheta] \times \mathbb R^n$: $\max\{\|x\|,\|x'\|\} \leq \alpha$.
\end{lemma}

\begin{lemma}\label{lem:lambda_x_local}
Let $\alpha_0,\lambda_0 > 0$. There exists $\lambda_X = \lambda_X(\alpha_0, \lambda_0) > 0$ with the following property. Let $(\tau,z,w(\cdot)) \in \mathbb G_*$. Let $\delta_w > 0$ be such that $w(\cdot)$ is continuously differentiable on $[-h, -h + \delta_w]$ {\rm(}see {\em(\ref{def:G_G_s}))}. Let the relations
\begin{equation}\label{lem:lambda_x_local:condition}
(z,w(\cdot)) \in P(\alpha_0),\quad \|w(\xi) - w(\xi')\| \leq \lambda_0 |\xi - \xi'|,\quad \xi,\xi' \in [-h,-h + \delta_w],
\end{equation}
hold. Then, every function $x(\cdot) \in X^1(\tau,z,w(\cdot))$ satisfies the inequality
\begin{equation}\label{lem:lambda_x_local:statement}
\|x(t) - x(t')\| \leq \lambda_X |t - t'|,\quad t,t' \in [\tau, \tau + \delta_w].
\end{equation}
\end{lemma}
\begin{proof}
According to Lemmas \ref{lem:alpha_x_lambda_x}, \ref{lem:g_s}, define $\lambda_X^* = \lambda_X^*(\alpha_0)$ and $\lambda_g = \lambda_g(\alpha_0)$. Put $\lambda_X = \lambda_X^* + \lambda_g (1+ \lambda_0)$. Then the inequality (\ref{lem:lambda_x_local:statement}) follows from the estimates
\begin{equation*}
\|x(t) - x(t')\| \leq \|y(t) - y(t')\| + \|g(t,w(t-\tau-h)) - g(t',w(t'-\tau-h))\| \leq \lambda_X (t - t').
\end{equation*}
The lemma has been proved.
\end{proof}

Let $(z,w(\cdot)) \in \mathbb R^n \times \mathrm{PLip}$. Denote by $\Gamma(z,w(\cdot))$ the set of sequences $\{w^j(\cdot)\}_{j \in \mathbb N} \subset \mathrm{Lip}$ such that
\begin{equation}\label{def:Gamma}
\begin{array}{c}
\|w(\cdot) - w^j(\cdot)\|_1 \to 0,\quad \|z - w^j(-0)\| \to 0,\\[0.2cm]
\|w(\xi) - w^j(\xi)\| \to 0,\quad \xi \in [-h,0),
\end{array}
\quad \text{as } j \to \infty.
\end{equation}

\begin{lemma}\label{lem:Gamma_nonempty}
For each $(z,w(\cdot)) \in \mathbb R^n \times \mathrm{PLip}$, there exists a sequence $\{w^j(\cdot)\}_{j \in \mathbb N} \in \Gamma(z,w(\cdot))$ such that
\begin{description}
\item[$(w^j_1)$] The inclusion $w^j(\cdot) \in \mathrm{C}^1$ holds for any $j \in \mathbb N$.
\vspace{0.1cm}

\item[$(w^j_2)$] The inequality $\|w^j(\cdot)\|_\infty \leq \max\{\|z\|,\|w(\cdot)\|_\infty\}$ holds for any $j \in \mathbb N$.
\vspace{0.1cm}

\item[$(w^j_3)$] For every $\xi \in (-h,0)$ satisfying the equality $w(\xi-0) = w(\xi)$, there exists $\delta > 0$ such that
\begin{equation*}
\max\limits_{\zeta \in [-\delta,\delta]} \|w(\xi+\zeta) - w^j(\xi+\zeta)\| \to 0 \text{ as } j \to \infty.
\end{equation*}
\end{description}
\end{lemma}
\begin{proof}
Let us take a continuously differentiable function $\beta(\cdot) \colon (- \infty, + \infty) \mapsto \mathbb [0,+\infty)$ such that
$$
\beta(\xi) = 0,\quad \xi \in (-\infty,0]\cup [1,+\infty),\quad \int_{-\infty}^{+\infty} \beta(\zeta) \mathrm{d}\zeta = 1.
$$
Let $\overline{w}(\cdot)\colon [-h,1] \mapsto \mathbb R^n$ be such that $\overline{w}(\xi) = w(\xi)$, $\xi \in [-h,0)$ and $\overline{w}(\xi) = z$, $\xi \in [0,1]$. Then one can show the functions
$$
w^j(\xi) = j \int_{-\infty}^{+\infty} \beta(j \zeta) \overline{w}(\xi + \zeta) \mathrm{d} \zeta,\quad j \in \mathbb N,
$$
satisfy the statements of the lemma.
\end{proof}

\begin{lemma}\label{lem:Gamma_motion}
Let $(\tau,z,w(\cdot)) \in \mathbb G$. Let $\{w^j(\cdot)\}_{j \in \mathbb N} \subset \Gamma(z,w(\cdot))$ satisfy conditions $(w^j_1)$--$(w^j_3)$. Let $x^j(\cdot) \in X^{1/j}(\tau,w^j(-0),w^j(\cdot))$, $j \in \mathbb N$. Then there exist a subsequence $x^{\,j_m}(\cdot)$, $m \in \mathbb N$ and a function $x(\cdot) \in X^0(\tau,z,w(\cdot))$ such that $\{x^{j_m}_t(\cdot)\}_{m \in \mathbb N} \in \Gamma(x(t),x_t(\cdot))$ for any $t \in [\tau, \min\{\tau + h,\vartheta\})$.
\end{lemma}
\begin{proof}
Let $\alpha_0 = \max\{\|z\|, \|w(\cdot)\|_\infty\}$. In accordance with Lemma \ref{lem:alpha_x_lambda_x}, define $\alpha^*_X = \alpha^*_X(\alpha_0)$ and $\lambda^*_X = \lambda^*_X(\alpha_0)$. Then, for the functions $y^j(t) = x^j(t) - g(t,x^j(t-h))$, $t \in [\tau,\vartheta]$, $j \in \mathbb N$, we have
\begin{equation}\label{lem:Gamma_motion:yi}
\|y^j(t)\| \leq \alpha_X^*,\quad \|y^j(t) - y^j(t')\| \leq \lambda_X^* |t - t'|,\quad t,t' \in [\tau,\vartheta].
\end{equation}
Due to these estimates and Arzela-Ascoli theorem (see, e.g., \cite[p. 207]{Natanson_1960}), without loss of generality, we can suppose that there exists a continuous function $y(\cdot) \colon [\tau,\vartheta] \mapsto \mathbb R^n$ such that
\begin{equation}\label{lem:Gamma_motion:y}
\max\limits_{t \in [\tau,\vartheta]} \|y(t) - y^j(t)\| \to 0\text{ as } j \to \infty.
\end{equation}
From (\ref{lem:Gamma_motion:yi}) and (\ref{lem:Gamma_motion:y}), one can establish the Lipschitz continuity of $y(\cdot)$. Denote $\tau_h = \min\{\tau + h,\vartheta\}$. Define the function $x(\cdot)$ so that
\begin{equation}\label{lem:Gamma_motion:x}
x(t) =
\left\{
\begin{array}{ll}
w(t-\tau), \text{if } t \in [\tau-h, \tau),\\[0.2cm]
y(t) + g(t,w(t-\tau-h)),\ \text{if } t \in [\tau, \tau_h),\\[0.2cm]
g(t,x(t-h)),\ \text{if } t \in [\tau_h, \vartheta].
\end{array} \right.
\end{equation}
%Let $t \in [\tau, \tau_h]$.
Choose $\lambda_g = \lambda_g(\alpha_0)$ according to Lemma \ref{lem:g_s}. Then, taking into account definitions of $x^j(\cdot)$, $y^j(\cdot)$, and $x(\cdot)$ and condition $(w^j_2)$, we have
\begin{equation*}
\|x(\xi) - x^j(t)\| \leq \|y(t) - y^j(t)\| + \lambda_g \|w(t-\tau-h) - w^j(t-\tau-h)\|,\quad \xi \in [\tau, \tau_h).
\end{equation*}
\begin{equation}\label{lem:Gamma_motion:lambda_g}
\|x(t) - x^j(t)\| = \|w(t - \tau) - w^j(t - \tau)\|,\quad t \in [\tau-h, \tau),
\end{equation}
From these estimates, taking into account the inclusion $\{w^j(\cdot)\}_{j \in \mathbb N} \in \Gamma(z, w(\cdot))$ and (\ref{lem:Gamma_motion:y}), we obtain $\{x^j_t(\cdot)\}_{j \in \mathbb N} \in \Gamma(x(t), x_t(\cdot))$ for any $t \in [\tau, \tau_h)$.

Let us show the inclusion $x(\cdot) \in X^0(\tau,z,w(\cdot))$. Firstly, according to (\ref{lem:Gamma_motion:x}), (\ref{lem:Gamma_motion:lambda_g}), the inclusion $x(\cdot) \in \Lambda(\tau,z,w(\cdot))$ holds. Note that, due to (\ref{lem:Gamma_motion:x}), the function $x(\cdot)$ satisfies (\ref{def:inclusion}) for every $t \in (\tau_h, \vartheta]$ in which $\eta = 0$. Let $t \in (\tau, \tau_h)$ be such that there exists $\mathrm{d} y(t)/ \mathrm{d} t$ and $w(t-\tau-h-0) = w(t-\tau-h)$. Let $\varepsilon > 0$. Then, due to the inclusion $\{w^j(\cdot)\}_{j \in \mathbb N} \in \Gamma(z,w(\cdot))$ and relations (\ref{lem:Gamma_motion:yi}), (\ref{lem:Gamma_motion:y}), there exist $\delta \in (0,\min\{t-\tau, \tau_h-t, \varepsilon / (9 c_H \lambda_g)\})$ and $j_* > 0$ such that
\begin{equation}\label{lem:Gamma_motion:delta}
\begin{array}{c}
\|w(t-\tau-h) - w^j(t-\tau-h+\zeta)\| \leq \varepsilon / (3 c_H\max\{1, 3 \lambda_g\}),\\[0.2cm]
\|y(t) - y^j(t + \zeta)\| \leq \varepsilon / (9 c_H),\quad 1/j \leq \varepsilon / 3,
\end{array}
\end{equation}
for any $\zeta \in [-\delta,\delta]$ and $j \geq j_*$, where $c_H$ is from condition $(H_2)$. From these estimates, choice of $\lambda_g$, and definitions of $x^j(\cdot)$, $y^j(\cdot)$, and $x(\cdot)$, we derive
\begin{equation*}
\begin{array}{c}
\|x(t) - x^j(t + \zeta)\| \leq \|y(t) - y^j(t + \zeta)\| \\[0.2cm]
+  \lambda_g \big(\zeta + \|w(t-\tau-h) - w^j(t-\tau-h+\zeta)\|\big)
\leq \varepsilon / (3 c_H).
\end{array}
\end{equation*}
Then, due to the inclusion $x^j(\cdot) \in X^{1/j}(\tau,w^j(-0),w^j(\cdot))$, and the first and the third inequalities in (\ref{lem:Gamma_motion:delta}), we derive
\begin{equation*}
\|\mathrm{d}\, y^j(t + \zeta) / \mathrm{d} t\| \leq c_H \big(\|x(t)\| + \|w(t-\tau-h)\|\big)  + \varepsilon\ \text{for a.e. }\ \zeta \in [-\delta,\delta].
\end{equation*}
Using Lemma 12 from \cite[p. 63]{Filippov_1988}, we obtain
\begin{equation*}
\|\mathrm{d} y(t) / \mathrm{d} t\|\leq c_H\big(\|x(t)\| + \|w(t-\tau-h)\|\big) + \varepsilon.
\end{equation*}
This estimate holds for every $\varepsilon > 0$ that means for $\varepsilon = 0$. Thus, the inclusion $x(\cdot) \in X^0(\tau,z,w(\cdot))$ is proved.
\end{proof}

\begin{lemma}\label{lem:dg}
Let $(\tau,z,w(\cdot)) \in \mathbb G_*$. Let $\delta_w \in (0,h)$ be such that $w(\cdot)$ is continuously differentiable on $[-h, -h + \delta_w]$. Then the inclusion $(t,x(t),x_t(\cdot)) \in \mathbb G_*$ holds for any $t \in [\tau, \tau + \delta_w]$ {\rm(}see {\em(\ref{def:G_G_s}))}, and
the following functions coincide and are continuous:
\begin{equation*}
\overline{g}(t) := \frac{\mathrm{d}}{\mathrm{d} t} \big(g(t,x(t-h)\big) = \partial^{ci}_{\tau,w} g(t,x_t(\cdot)) = \partial^{ci}_{\tau,w} g(t,\kappa_t(\cdot)),\quad t \in [\tau, \tau + \delta_w].
\end{equation*}
for any $x(\cdot) \in \Lambda(\tau,z,w(\cdot))$ and $\kappa(\cdot) \in \Lambda_0(\tau,z,w(\cdot))$.
\end{lemma}
\begin{proof}
is directly following from definition (\ref{def:derivative_g}) of $\partial^{ci}_{\tau,w} g(\tau,w(\cdot))$.
\end{proof}

\begin{lemma}\label{lem:kappa}
Let $(z,w(\cdot)) \in \mathbb R^n \times \mathrm{PLip}$. Then, for every $\varepsilon > 0$, there exists $\nu = \nu(\varepsilon) > 0$ such that, for every $\tau \in [0,\vartheta]$ and $t, t' \in [\tau,\vartheta]$ satisfying $|t - t'| \leq \nu$ the inequality below holds:
\begin{displaymath}
\|\kappa_t(\cdot) - \kappa_{t'}(\cdot)\|_1 \leq \varepsilon,\quad \kappa(\cdot) \in \Lambda_0(\tau,z,w(\cdot)).
\end{displaymath}
\end{lemma}
\begin{proof}
can be obtain using approximation of $(z,w(\cdot))$ by Lipschitz continuous functions (see Lemma \ref{lem:Gamma_nonempty}).
\end{proof}

\subsection{Proof of Theorem \ref{teo:Lip_minimax_solution}}

According to Theorem 3 from \cite{Plaksin_2021b}, the following lemma holds:
\begin{lemma}\label{lem:ex_un_minimax_sol}
There exists a unique continuous (with respect to uniform norm) functional $\hat{\varphi} \colon [0,\vartheta] \times \mathrm{Lip} \mapsto \mathbb R$  satisfying condition
$\varphi(\vartheta,w(\cdot)) = \sigma(w(-0),w(\cdot))$ for any $w(\cdot) \in \mathrm{Lip}$
and the inequalities
\begin{subequations}
\begin{align}
\inf\limits_{x(\cdot) \in X^0(\tau,w(-0),w(\cdot))} \big(\hat{\varphi}(t,x_t(\cdot)) + \omega(\tau,t,x(\cdot),s)\big) \leq \hat{\varphi}(\tau,w(\cdot)), \label{def:upper_minmax_solution_} \\[0.0cm]
\sup\limits_{x(\cdot) \in X^0(\tau,w(-0),w(\cdot))} \big(\hat{\varphi}(t,x_t(\cdot)) + \omega(\tau,t,x(\cdot),s)\big) \geq \hat{\varphi}(\tau,w(\cdot)), \label{def:lower_minmax_solution_}
\end{align}
\end{subequations}
for any $(\tau,w(\cdot)) \in [0,\vartheta) \times \mathrm{Lip}$, $t \in (\tau,\vartheta]$, and $s \in \mathbb R^n$, where $\omega$ is from {\rm (\ref{def:omega})}.
\end{lemma}

Lemmas \ref{lem:lk}, \ref{lem:upper_lower_stability} below can be proved similar to Lemmas 1, 3 from \cite{Plaksin_2021b}.

Let $\alpha > 0$. Define $\lambda_g = \lambda_g(\alpha) > 1$ and $\lambda_H = \lambda_H(\alpha) > 1$ according to Lemma \ref{lem:g_s} and condition $(H_3)$. For every $\gamma,\varepsilon > 0$ and $(\tau,z,w(\cdot)) \in \mathbb G$, denote
\begin{equation}\label{def:nu_theta_kappa}
\begin{array}{rcl}
\theta^\alpha_\gamma(\tau) &=& (e^{-(4 \lambda_H + 2 \lambda_g / h) \tau} - \gamma) / \gamma,\\[0.2cm]
\nu^\alpha_{\gamma,\varepsilon}(\tau,z,w(\cdot))
&=& \theta^\alpha_\gamma(\tau)
\bigg(\!\sqrt{\varepsilon^4 + \|z\|^2}
+ 2 \lambda_H \displaystyle\int_{-h}^0\!\! \bigg(1 - \frac{2 \lambda_g \xi}{h}\bigg) \|w(\xi)\| \mathrm{d}\xi \! \bigg).
\end{array}
\end{equation}

\begin{lemma}\label{lem:lk}
Let $\alpha,\varepsilon > 0$ and $\tau \in [0,\vartheta]$. Let $\gamma > 0$ be such that $\theta^\alpha_\gamma(t) > 1$ for any $t \in [0,\vartheta]$. Let $x(\cdot), x'(\cdot)$ be Lipschitz continuous functions from $[\tau-h,\vartheta]$ to $\mathbb R^n$ satisfying the inequality
\begin{equation}\label{lem:lk:condition}
\|x(t)\| \leq \alpha,\quad \|x'(t)\| \leq \alpha,\quad t \in [\tau-h,\vartheta].
\end{equation}
Then the following inequality holds:
\begin{equation*}
\nu^\alpha_{\gamma,\varepsilon}(t,\Delta y(t),\Delta x_t(\cdot)) - \nu^\alpha_{\gamma,\varepsilon}(\tau,\Delta y(\tau),\Delta x_\tau(\cdot))
\leq \int_\tau^t \Delta H^\alpha_{\gamma,\varepsilon}(\xi) \mathrm{d} \xi,\quad t \in [\tau,\vartheta].
\end{equation*}
where $\Delta x(t) = x(t) - x'(t)$, $\Delta y(t) = \Delta x(t) - g(t,x(t-h)) + g(t,x'(t-h))$, and
\begin{equation*}
\begin{array}{rcl}
\Delta H^\alpha_{\gamma,\varepsilon}(t) &=& H(t,x(t),x(t-h),\nabla_z \nu^\alpha_{\gamma,\varepsilon}(t,\Delta y(t),\Delta x_t(\cdot))) \\[0.2cm]
&-& H(t,x'(t),x'(t-h),\nabla_z \nu^\alpha_{\gamma,\varepsilon}(t,\Delta y(t),\Delta x_t(\cdot)))  \\[0.2cm]
&+& \langle \Delta y(t), \nabla_z \nu^\alpha_{\gamma,\varepsilon}(t,\Delta y(t),\Delta x_t(\cdot)) \rangle.
\end{array}
\end{equation*}
\end{lemma}

\begin{lemma}\label{lem:upper_lower_stability}
Let $\hat{\varphi} \colon [0,\vartheta] \times \mathrm{Lip} \mapsto \mathbb R$ be taken from Lemma \ref{lem:ex_un_minimax_sol}. Let $\alpha, \gamma, \varepsilon > 0$, $(\tau,w(\cdot)), (\tau,w'(\cdot)) \in [0,\vartheta] \times \mathrm{Lip}$ and $t \in [\tau,\vartheta]$. Then there exist functions  $x(\cdot) \in X^0(\tau,w(-0),w(\cdot))$ and $x'(\cdot) \in X^0(\tau,w'(-0),w'(\cdot))$ such that
\begin{equation*}
\begin{array}{c}
\hat{\varphi}(t,x_t(\cdot)) - \hat{\varphi}(t,x'_t(\cdot))
+ \displaystyle\int_\tau^t \Delta H_{\gamma,\varepsilon}^\alpha(\xi) \mathrm{d} \xi \leq \hat{\varphi}(\tau,w(\cdot)) - \hat{\varphi}(\tau,w'(\cdot)) + (t - \tau) \varepsilon,
\end{array}
\end{equation*}
where we use denotations from Lemma \ref{lem:lk}.
\end{lemma}

\begin{lemma}\label{lem:phi_lip}
Let $\hat{\varphi} \colon [0,\vartheta] \times \mathrm{Lip} \mapsto \mathbb R$ be taken from Lemma \ref{lem:ex_un_minimax_sol}. For every $\alpha > 0$, there exists $\lambda_\varphi = \lambda_\varphi(\alpha) > 0$ such that
\begin{equation*}
|\hat{\varphi}(\tau,w(\cdot)) - \hat{\varphi}(\tau,w'(\cdot))| \leq \lambda_\varphi \upsilon(\tau,w(-0) - w'(-0),w(\cdot) - w'(\cdot))
\end{equation*}
for any $\tau \in [0,\vartheta]$ and $(z,w(\cdot)), (z',w'(\cdot)) \in P(\alpha) \cap (\mathbb R^n \times \mathrm{Lip})$, where $P(\alpha)$ and $\upsilon$ are defined according to {\rm (\ref{def:P})} and {\rm (\ref{def:upsilon})}, respectively.
\end{lemma}
\begin{proof}
Let us prove that, for each $i \in \overline{0,I}$ and $\alpha > 0$, there exists $\lambda_i = \lambda_i(\alpha) > 0$ such that
\begin{equation}\label{lem:phi_lip:induction}
\begin{array}{c}
\hat{\varphi}(\tau,w(\cdot)) - \hat{\varphi}(\tau,w'(\cdot)) \leq \lambda_i \upsilon(\tau,w(-0) - w'(-0),w(\cdot) - w'(\cdot))
\end{array}
\end{equation}
for any $\tau \in [i h,\vartheta]$ and $(w(-0),w(\cdot)), (w'(-0),w'(\cdot)) \in P(\alpha) \cap (\mathbb R^n \times \mathrm{Lip})$. After that, taking $\lambda_\varphi = \lambda_0$, we will get the statement of the lemma.

Note that, for $i = I$ and each $\alpha > 0$, inequality (\ref{lem:phi_lip:induction}) holds due to (\ref{def:upsilon}), conditions (\ref{terminal_condition}), and $(\sigma)$ if we take $\lambda_I = \lambda_I(\alpha) = \lambda_\sigma(\alpha)$

Assume that inequality (\ref{lem:phi_lip:induction}) holds for $i = j + 1 \leq I$ and prove it for $i = j$. Let $\alpha > 0$. According to Lemmas \ref{lem:alpha_x_lambda_x}, \ref{lem:g_s} and condition $(H_2)$, define $\alpha_X = \alpha_X(\alpha)$, $\lambda_g = \lambda_g(\alpha_X)$, and $\lambda_H = \lambda_H(\alpha_X) > 1$. Due to our assumption, there exists $\lambda_{j+1} = \lambda_{j+1}(\alpha_X)$ such that (\ref{lem:phi_lip:induction}) holds for $i = j + 1$. In accordance with (\ref{def:nu_theta_kappa}), there exist $\gamma > 0$ such that
\begin{equation}\label{lem:phi_lip:gamma}
\theta^{\alpha_X}_{\gamma}(0) \geq \theta^{\alpha_X}_{\gamma}(t) \geq \max\{\lambda_{j+1},1\},\quad t \in [0,\vartheta].
\end{equation}
Put
\begin{equation}\label{lem:phi_lip:lambda_j}
\lambda_j = 2 \theta^{\alpha_X}_{\gamma}(0) \lambda_H (1 + 2 \lambda_g) + \lambda_{j+1} (2 + \lambda_g).
\end{equation}
Since $\lambda_j > \lambda_{j+1}$, then inequality (\ref{lem:phi_lip:induction}) already holds for $\lambda_j$ and $\tau \in [(j+1)h, \vartheta]$. Let $\tau \in [j h,(j+1) h)$ and $(w(-0),w(\cdot)), (w'(-0),w'(\cdot)) \in P(\alpha) \cap (\mathbb R^n \times \mathrm{Lip})$. Let us show, for every $\zeta > 0$, the following estimate:
\begin{equation}\label{lem:phi_lip:zeta}
\hat{\varphi}(\tau,w'(\cdot)) - \hat{\varphi}(\tau,w(\cdot)) \leq \lambda_j \upsilon(\tau,w'(-0) - w(-0),w'(\cdot) - w(\cdot)) + \zeta.
\end{equation}
Let $\zeta > 0$. Denote $\vartheta_* = (j + 1) h$. Choose $\varepsilon > 0$ such that
\begin{equation}\label{lem:phi_lip:epsilon}
\theta^{\alpha_X}_{\gamma}(\tau) \varepsilon^2 + (\vartheta_* - \tau) \varepsilon \leq \zeta.
\end{equation}
According to Lemma \ref{lem:upper_lower_stability}, where we take $\alpha = \alpha_X$, define the functions $x(\cdot) \in X^0(\tau,w(-0),w(\cdot))$ and $x'(\cdot) \in X^0(\tau,w'(-0),w'(\cdot))$. Due to the choice of $\alpha_X$, these functions satisfy (\ref{lem:lk:condition}) for $\alpha = \alpha_X$. Then, using Lemma \ref{lem:lk}, we have
\begin{equation}\label{lem:phi_lip:1}
\begin{array}{c}
\hat{\varphi}(\tau,w'(\cdot)) - \hat{\varphi}(\tau,w(\cdot)) \leq \hat{\varphi}(\vartheta_*,x'_{\vartheta_*}(\cdot)) - \hat{\varphi}(\vartheta_*,x_{\vartheta_*}(\cdot)) \\[0.2cm]
+ \nu^{\alpha_X}_{\gamma,\varepsilon}(\tau,\Delta y(\tau),\Delta x_\tau(\cdot))
- \nu^{\alpha_X}_{\gamma,\varepsilon}(\vartheta_*,\Delta y(\vartheta_*),\Delta x_{\vartheta_*}(\cdot)) + (\vartheta_* - \tau) \varepsilon.
\end{array}
\end{equation}
Due to the choice of $\lambda_g$, $\lambda_H > 1$, and (\ref{def:upsilon}), (\ref{def:nu_theta_kappa}), (\ref{lem:lk:condition}), (\ref{lem:phi_lip:gamma}), we derive
\begin{equation}\label{lem:phi_lip:2}
\begin{array}{c}
\nu^{\alpha_X}_{\gamma,\varepsilon}(\tau,\Delta y(\tau),\Delta x_\tau(\cdot))\\[0.2cm]
 \leq \theta^{\alpha_X}_\gamma(\tau) \big(\varepsilon^2 + \|\Delta x(\tau)\| + \lambda_g \|\Delta x(\tau-h)\|
+ 2 \lambda_H (1 + 2 \lambda_g) \|\Delta x_\tau(\cdot)\|_1\big) \\[0.2cm]
\leq \theta^{\alpha_X}_\gamma(\tau) \varepsilon^2 + 2 \theta^{\alpha_X}_\gamma(0) \lambda_H( 1 + 2 \lambda_g) \upsilon(\tau,w(-0)-w'(-0),w(\cdot)-w'(\cdot))
\end{array}
\end{equation}
and, taking taking into account the choice of $\lambda_{j+1}$ and (\ref{lem:phi_lip:gamma}), we obtain
\begin{equation}\label{lem:phi_lip:3}
\begin{array}{c}
\hat{\varphi}(\vartheta_*,x'_{\vartheta_*}(\cdot)) - \hat{\varphi}(\vartheta_*,x_{\vartheta_*}(\cdot)) \leq
\lambda_{j+1} \upsilon(\vartheta_*,\Delta x(\vartheta_*), \Delta x_{\vartheta_*}(\cdot))\\[0.2cm]
\leq \lambda_{j+1} \big(\|\Delta y(\vartheta_*)\| + \|\Delta x_{\vartheta_*}(\cdot)\|_1
+ (2 + \lambda_g) \|\Delta x(j h)\|\big) \\[0.2cm]
\leq \nu^{\alpha_X}_{\gamma,\varepsilon}(\vartheta_*,\Delta y(\vartheta_*),\Delta x_{\vartheta_*}(\cdot)) + \lambda_{j+1} (2 + \lambda_g) \|\Delta x(j h)\|.
\end{array}
\end{equation}
Due to (\ref{def:upsilon}), the inequality $\|\Delta x(j h)\| \leq \upsilon(\tau,w(-0)-w'(-0),w(\cdot)-w'(\cdot))$ in both cases of $\tau = j h$ and $\tau \in (j h,\vartheta_*)$. Thus, from (\ref{lem:phi_lip:lambda_j}),  (\ref{lem:phi_lip:epsilon})--(\ref{lem:phi_lip:3}), we get (\ref{lem:phi_lip:zeta}).
\end{proof}

\begin{lemma}\label{lem:varphi_extension}
Let $\hat{\varphi} \colon [0,\vartheta] \times \mathrm{Lip} \mapsto \mathbb R$ be taken from Lemma \ref{lem:ex_un_minimax_sol}. Then there exists a unique functional $\varphi \colon \mathbb G \mapsto \mathbb R$ satisfying conditions $(\varphi_1)$, $(\varphi_2)$ and the equality
\begin{equation}\label{lem:varphi_extension:statement}
\varphi(\tau,w(-0),w(\cdot)) = \hat{\varphi}(\tau,w(\cdot)),\quad
(\tau,w(\cdot)) \in [0,\vartheta] \times \mathrm{Lip}.
\end{equation}
\end{lemma}
\begin{proof}
Let $(\tau,z,w(\cdot)) \in \mathbb G$ and $\alpha_0 = \{\|z\|,\|w(\cdot)\|_\infty\}$. Due to Lemma \ref{lem:Gamma_nonempty}, there exists $\{w^j(\cdot)\}_{j \in \mathbb N} \in \Gamma(z,w(\cdot))$ satisfying condition $(w^j_2)$. Let us consider the sequence $A^j = \hat{\varphi}(\tau,w^j(\cdot))$, $j \in \mathbb N$. Then, taking $\lambda_\varphi = \lambda_\varphi(\alpha_0)$ according to Lemma \ref{lem:phi_lip} and (\ref{def:upsilon}), we have
\begin{equation*}
|A^j| \leq |\hat{\varphi}(\tau,w^j(\cdot)) - \hat{\varphi}(\tau,0(\cdot))| + |\hat{\varphi}(\tau,0(\cdot))| \leq \lambda_\varphi (3 + h) \alpha_0 + |\hat{\varphi}(\tau,0(\cdot))|.
\end{equation*}
Therefore, the sequence $A^j$ is bounded. Hence, without loss of generality, we can assume the existence of $A^*$ such that $A^{j} \to A^*$ as $j \to \infty$. Let us show  this limit does not depend on the choice of $\{w^j(\cdot)\}_{j \in \mathbb N} \in \Gamma(z,w(\cdot))$ satisfying condition $(w^j_2)$. Let $\{r^j(\cdot)\}_{j \in \mathbb N} \in \Gamma(z,w(\cdot))$ satisfy condition $(w^j_2)$. Then, for defined above $\lambda_\varphi$, using the definition of $\Gamma(z,w(\cdot))$ (see (\ref{def:Gamma})) and (\ref{def:upsilon}), we derive
\begin{equation*}
|\hat{\varphi}(\tau,w^j(\cdot)) - \hat{\varphi}(\tau,r^j(\cdot))| \leq \lambda_\varphi \upsilon(\tau, w^j(-0) - r^j(-0), w^j(\cdot) - r^j(\cdot)) \to 0
\end{equation*}
as $j \to +\infty$. Thus, $\hat{\varphi}(\tau,r^j(\cdot)) \to A^*$ as $j \to +\infty$ for any $\{r^j(\cdot)\}_{j \in \mathbb N} \in \Gamma(z,w(\cdot))$ satisfying condition $(w^j_2)$. Put $\varphi(\tau,z,w(\cdot)) = A^*$. By the similar way, we can define the values $\varphi(\tau,z,w(\cdot))$ for any $(\tau,z,w(\cdot)) \in \mathbb G$.

The equality (\ref{lem:varphi_extension:statement}) directly follows from the inclusion
$\{w^j(\cdot) \equiv w(\cdot)\}_{j \in \mathbb N} \in \Gamma(w(-0),w(\cdot))$ for any $w(\cdot) \in \mathrm{Lip}$. Condition $(\varphi_1)$ holds for $\varphi$ due to (\ref{lem:varphi_extension:statement}) and the continuity of $\hat{\varphi}$.

Let us show the condition $(\varphi_2)$. Let $\alpha > 0$. Define $\lambda_\varphi = \lambda_\varphi(\alpha)$ according to Lemma \ref{lem:phi_lip}. Let $\tau \in [0,\vartheta]$ and $(z,w(\cdot)), (z',w'(\cdot)) \in P(\alpha)$. Due to Lemma~\ref{lem:Gamma_nonempty}, there exist $\{w^j(\cdot)\}_{j \in \mathbb N} \in \Gamma(z,w(\cdot))$ and $\{r^j(\cdot)\}_{j \in \mathbb N} \in \Gamma(z',w'(\cdot))$ satisfying condition $(w^j_2)$. Then we have
\begin{equation*}
|\hat{\varphi}(\tau,w^j(\cdot)) - \hat{\varphi}(\tau,r^j(\cdot))| \leq \lambda_\varphi \upsilon(\tau,w^j(-0)-r^j(-0),w^j(\cdot)-r^j(\cdot)),\quad j \in \mathbb N.
\end{equation*}
Passing to the limit as $j \to + \infty$, taking into account the construction of $\varphi$, we obtain (\ref{Phi_lip}).

Let us prove the uniqueness. Let functionals $\varphi \colon \mathbb G \mapsto \mathbb R$ and $\varphi' \colon \mathbb G \mapsto \mathbb R$ satisfy $(\varphi_1)$, $(\varphi_2)$, and (\ref{lem:varphi_extension:statement}). Let $(\tau,z,w(\cdot)) \in \mathbb G$. According to Lemma~\ref{lem:Gamma_nonempty}, there exists $\{w^j(\cdot)\}_{j \in \mathbb N} \in \Gamma(z,w(\cdot))$ satisfying condition $(w^j_2)$. Let $\alpha_0 = \{\|z\|,\|w(\cdot)\|_\infty\}$. Then, due to $(\varphi_2)$, there exists $\lambda_\varphi(\alpha_0)$ such that, taking (\ref{lem:varphi_extension:statement}) into account, the following estimates holds:
\begin{equation*}
\begin{array}{c}
|\varphi(\tau,z,w(\cdot)) - \varphi'(\tau,z,w(\cdot))|
\leq |\varphi(\tau,z,w(\cdot)) - \varphi(\tau,w^j(-0),w^j(\cdot))| \\[0.2cm]
+ |\varphi'(\tau,z,w(\cdot)) - \varphi'(\tau,w^j(-0),w^j(\cdot))| \leq 2 \lambda_\varphi \upsilon(\tau,z - w^j(-0),w(\cdot) - w^j(\cdot)).
\end{array}
\end{equation*}
Passing to the limit as $j \to + \infty$, we get $\varphi(\tau,z,w(\cdot)) = \varphi'(\tau,z,w(\cdot))$.
\end{proof}

\subsection{Proof of Theorem \ref{teo:minimax_solutions}}

\begin{lemma}\label{lem:C1_and_minimax_solutions}
If the functional $\varphi \colon \mathbb G \mapsto \mathbb R$ is a $\mathrm{C}^1$-minimax solution of problem {\rm (\ref{Hamilton-Jacobi_equation}), (\ref{terminal_condition})}, then it is the minimax solution of problem {\rm (\ref{Hamilton-Jacobi_equation}), (\ref{terminal_condition})}.
\end{lemma}
\begin{proof}
Let us show that $\varphi$ satisfies  inequalities (\ref{def:upper_minmax_solution}), (\ref{def:lower_minmax_solution}) for any $(\tau,z,w(\cdot)) \in \mathbb G$, $\tau < \vartheta$, $t \in (\tau,\vartheta]$, $s \in \mathbb R^n$, and $\eta = 0$, where, without loss of generality, we can suppose that $i h \leq \tau < t \leq (i+1) h$ for some $i \in \overline{0,I-1}$ and $t < \tau + h$.

Let $i \in \overline{0,I-1}$, $(\tau,z,w(\cdot)) \in [ih, (i+1)h) \times \mathbb R^n \times \mathrm{PLip}$, $t \in (\tau,(i+1)h]$, and $s \in \mathbb R^n$. Due to Lemma \ref{lem:Gamma_nonempty}, there exists a sequence $\{w^j(\cdot)\}_{j \in \mathbb N} \in \Gamma(z,w(\cdot))$ such that conditions $(w^j_1)$--$(w^j_3)$ hold. Since $\varphi$ is a $\mathrm{C}^1$-minimax solution of problem (\ref{Hamilton-Jacobi_equation}), (\ref{terminal_condition}), for each $j \in \mathbb N$, there exists $x^j(\cdot) \in X^{1/j}(\tau,w^j(-0),w^j(\cdot))$ such that
\begin{equation}\label{lem_cb:1}
\varphi(t,x^j(t),x^j_t(\cdot)) + \omega(\tau,t,x^j(\cdot),s) \leq \varphi(\tau,w^j(-0),w^j(\cdot)) + 1 / j.
\end{equation}
In accordance with Lemma \ref{lem:Gamma_motion}, without loss of generality, we can suppose the existence of $x(\cdot) \in X^0(\tau,z,w(\cdot))$ such that $\{x^j_t(\cdot)\}_{j \in \mathbb N} \in \Gamma(x(t),x_t(\cdot))$. Let $\alpha_0 = \max\{\|z\|,\|w(\cdot)\|_\infty\}$. According Lemmas \ref{lem:alpha_x_lambda_x}, \ref{lem:g_s} and conditions $(H_3)$, $(\varphi_2)$, define $\alpha_X = \alpha_X(\alpha_0)$, $\lambda_g = \lambda_g(\alpha_X)$, $\lambda_H = \lambda_H(\alpha_X)$, and $\lambda_\varphi = \lambda_\varphi(\alpha_X)$. Denote $\lambda_\omega = \lambda_H (1 + \|s\|) + (1 + \lambda_g) \|s\|$. Then, using (\ref{def:omega}) and (\ref{def:Gamma}), we derive
\begin{equation*}
\begin{array}{rcl}
\big|\varphi(\tau,z,w(\cdot)) - \varphi(\tau,w^j(-0),w^j(\cdot))\big|
&\leq& \lambda_\varphi \upsilon(\tau,z-w^j(-0),w(\cdot)-w^j(\cdot)),\\[0.2cm]
\big|\varphi(t,x(t),x_t(\cdot)) - \varphi(t,x^j(t),x^j_t(\cdot))\big|
&\leq& \lambda_\varphi \upsilon(t,x(t)-x^j(t),x_t(\cdot)-x^j_t(\cdot)),\\[0.2cm]
\big|\omega(\tau,t,x(\cdot),s) - \omega(\tau,t,x^j(\cdot),s)\big|
&\leq& \lambda_\omega \upsilon(t,x(t)-x^j(t),x_t(\cdot)-x^j_t(\cdot)),\\[0.2cm]
&+& \lambda_\omega \upsilon(\tau,z-w^j(-0),w(\cdot)-w^j(\cdot)).
\end{array}
\end{equation*}
Thus, passing to the limit as $j \to + \infty$, form (\ref{lem_cb:1}), we derive
\begin{equation*}
\varphi(t,x(t),x_t(\cdot)) + \omega(\tau,t,x(\cdot),s) \leq \varphi(\tau,z,w(\cdot)).
\end{equation*}
This estimate implies (\ref{def:upper_minmax_solution}). Inequality (\ref{def:lower_minmax_solution}) can be proved by the similar way.

The functional $\varphi$ satisfies condition (\ref{terminal_condition}) for any $w(\cdot) \in \mathrm{PLip}$ since $\varphi$ satisfies condition (\ref{terminal_condition}) for any $w(\cdot) \in \mathrm{C}^1$ and due to Lemma (\ref{lem:Gamma_nonempty}) and condition $(\sigma)$.
\end{proof}

\subsection{Proof of Theorem \ref{teo:equivalent_solutions}}

\begin{lemma}\label{lem:delta_s}
Let $(\tau,z,w(\cdot)) \in \mathbb G_*$. Let $i \in \overline{0,I-1}$ satisfy $\tau \in (i h, (i + 1) h)$. Then there exists $\delta_* \in (0,h)$ such that $[\tau,\tau + \delta_*] \subset (i h, (i + 1) h)$ and the function $w(\cdot)$ is continuously differentiable on $[-h,-h+\delta_*]$.
\end{lemma}
\begin{proof}
According to definition (\ref{def:G_G_s}) of $G_*$, there exists $\delta_w > 0$ such that $w(\cdot)$ is continuously differentiable on $[-h,-h+\delta_w]$. Let $\delta_i > 0$ be such that $\tau + \delta_i < (i + 1) h$. Then, taking $\delta_* = \min\{\delta_w, \delta_i\}$, we obtain the statement of the lemma.
\end{proof}

\begin{lemma}\label{minimax_and_viscosity_solutions}
If the functional $\varphi \colon \mathbb G \mapsto \mathbb R$ is a minimax solution of problem {\rm (\ref{Hamilton-Jacobi_equation}), (\ref{terminal_condition})}, then it is the viscosity solution of problem {\rm (\ref{Hamilton-Jacobi_equation}), (\ref{terminal_condition})}.
\end{lemma}
\begin{proof}
Let us prove that $\varphi$ satisfy (\ref{sub_viscosity_solution}). Let $(\tau,z,w(\cdot)) \in \mathbb G_*$, $\psi \in \mathrm{C}^1(\mathbb R \times \mathbb R^n,\mathbb R)$, and  $\delta > 0$ be such that
\begin{equation*}
\varphi(\tau,z,w(\cdot)) - \psi(\tau,z) \leq \varphi(t,x,\kappa_t(\cdot)) - \psi(t,x),\quad (t,x) \in O^+_\delta(\tau,z),
\end{equation*}
where $\kappa(\cdot) \in \Lambda_0(\tau,z,w(\cdot))$. Let $\alpha_0 = \max\{\|z\|,\|w(\cdot)\|_\infty\}$. According to Lemmas \ref{lem:alpha_x_lambda_x}, \ref{lem:lambda_x_local}, \ref{lem:delta_s} and condition $(\varphi_2)$, define $\alpha_X = \alpha_X(\alpha_0)$, $\lambda_X = \lambda_X(\alpha_0)$, $\delta_*$, and  $\lambda_\varphi = \lambda_\varphi(\alpha_X)$, respectively. Denote $t_* = \tau + \min\{\delta_*, \delta, \delta / \lambda_X\}$. Then, we have
\begin{equation*}
\big|\varphi(t,x(t),x_{t}(\cdot)) - \varphi(t,x(t),\kappa_{t}(\cdot))\big| \leq \lambda_\varphi \displaystyle\int_\tau^{t} \|x(\xi) - z\| \mathrm{d} \xi \leq \lambda_\varphi \lambda_X (t - \tau)^2,
\end{equation*}
\begin{equation}
(t,x(t)) \in O^+_\delta(\tau,z),\quad t \in [\tau,t_*],\quad x(\cdot) \in X^1(\tau,z,w(\cdot)).
\end{equation}
Let $t_j = \tau + (t_* - \tau) / j$, $j \in \mathbb N$. Due to the inclusion $\psi \in \mathrm{C}^1(\mathbb R \times \mathbb R^n,\mathbb R)$, there exist $\varepsilon_j > 0$, $j \in \mathbb N$ such that $\varepsilon_j \to 0$ as $j \to \infty$ and
\begin{equation}
\frac{\psi(t_j,x(t_j)) - \psi(\tau,z)}{t_j - \tau} - \frac{\partial}{\partial \tau} \psi(\tau,z) - \langle \frac{x(t_j) - x(\tau)}{t_j - \tau}, \nabla_z \psi(\tau,z) \rangle \geq - \varepsilon_j
\end{equation}
for any $x(\cdot) \in X^1(\tau,z,w(\cdot))$. Since $\varphi$ is the minimax solution of problem (\ref{Hamilton-Jacobi_equation}), (\ref{terminal_condition}), for each $j \in \mathbb N$, there exists $x^j(\cdot) \in X^1(\tau,z,w(\cdot))$ such that
\begin{equation}
\varphi(t_j,x^j(t_j),x^j_{t_j}(\cdot)) + \omega(\tau,t_j,x^j(\cdot),\nabla_z \psi(\tau,z)) \leq \varphi(\tau,z,w(\cdot)) + (t_j - \tau) \varepsilon_j.
\end{equation}
Thus, taking (\ref{def:omega}) into account, we derive
\begin{equation*}
\begin{array}{c}
\displaystyle\frac{1}{t_j - \tau}\int_\tau^{t_j} H(\xi,x^j(\xi),x^j(\xi-h),\nabla_z \psi(\tau,z)) \mathrm{d} \xi \\[0.3cm]
\displaystyle + \frac{1}{t_j - \tau} \big\langle g(t_j,x(t_j-h)) - g(\tau,x(\tau-h)), \nabla_z \psi(\tau,z) \big\rangle\\[0.3cm]
\displaystyle\leq - \frac{\partial}{\partial \tau} \psi(\tau,z) + 2 \varepsilon_j + \lambda_\varphi \lambda_X (t_j - \tau).
\end{array}
\end{equation*}
Passing to the limit as $j \to \infty$, taking into accound condition $(H_1)$, inequalities  (\ref{lem:lambda_x_local:condition}), (\ref{lem:lambda_x_local:statement}), and Lemma \ref{lem:dg}, we obtain
\begin{equation*}
\partial \psi(\tau,z) / \partial \tau + \langle \partial^{ci}_{\tau,w} g(\tau,w(\cdot)), \nabla_z \psi(\tau,z) \rangle
+ H(\tau,z,w(\cdot),\nabla_z \psi(\tau,z)) \leq 0.
\end{equation*}
Thus, condition (\ref{sub_viscosity_solution}) is proved. Condition (\ref{super_viscosity_solution}) can be proved similarly.
\end{proof}

\begin{lemma}\label{lem:tilde_phi}
Let $\varphi \colon \mathbb G \mapsto \mathbb R$ satisfies $(\varphi_1)$, $(\varphi_2)$ and $(\tau,z,w(\cdot)) \in \mathbb G_*$. Let $\delta_* > 0$ be taken from Lemma \ref{lem:delta_s}.
Then the function $\tilde{\varphi}$ defined by
\begin{equation}\label{def:tilde_phi}
\tilde{\varphi}(t,x) = \varphi(t,x,\kappa_t(\cdot)),\quad \kappa(\cdot) \in \Lambda_0(\tau,z,w(\cdot)),
\end{equation}
is continuous at every $(t,x) \in [\tau,\tau + \delta_*] \times \mathbb R^n$.
\end{lemma}
\begin{proof}
Let $(t,x) \in [\tau,\tau + \delta_*] \times \mathbb R^n$ and $\varepsilon > 0$. Let $\alpha_0 = \max\{\|x\|,\|z\|,\|w(\cdot)\|_\infty\}$. Due to condition $(\varphi_2)$, define $\lambda_\varphi = \lambda_\varphi(\alpha_0 + 1)$. Let $\varepsilon_* = \varepsilon / (32 \lambda_\varphi)$. Then, applying Lemma \ref{lem:Gamma_nonempty} to $(x,\kappa_t)$, we obtain the existence of $w^*(\cdot) \in \mathrm{C}^1$ such that
\begin{equation}\label{lem:tilde_phi:overline_w}
\begin{array}{c}
\|w^*(\cdot)\|_\infty \leq \alpha_0,\quad \|x - w^*(-0)\| \leq \varepsilon_*,\quad \|\kappa_t(\cdot) - w^*(\cdot)\|_1 \leq \varepsilon_*,\\[0.3cm]
\|\kappa_t(-h) - w^*(-h)\| \leq \varepsilon_*,\quad
\|\kappa(i h) - w^*(i h - t)\| \leq \varepsilon_*.
\end{array}
\end{equation}
where $i \in \overline{0,I-1}$ satisfies $\tau \in (i h, (i + 1) h)$ which implies $t \in (i h, (i + 1) h)$ according to the choice $\delta_*$. Since $w^*(\cdot) \in \mathrm{C}^1$, there exist $\lambda^* > 0$ such that
\begin{equation}\label{lem:tilde_phi:overline_lambda}
\|w^*(\xi) - w^*(\xi')\| \leq \lambda^* |\xi - \xi|,\quad \xi,\xi' \in [-h, 0).
\end{equation}
Due to Lemma \ref{lem:kappa} and condition $(\varphi_1)$, there exists $\nu_* > 0$ such that
\begin{equation}\label{lem:tilde_phi:delta_s}
\|\kappa_t(\cdot) - \kappa_{t'}(\cdot)\|_1 \leq \varepsilon_*,\
|\varphi(t,w^*(-0),w^*(\cdot)) - \varphi(t',w^*(-0),w^*(\cdot))| \leq \varepsilon/2
\end{equation}
for every $t' \in [\tau, \tau + \delta_*]$: $|t' - t| \leq \nu_*$. Due to the choice of $\delta_*$, there exists $\lambda_0 > 0$ such that
\begin{equation}\label{lem:tilde_phi:lambda_0}
\|w(\xi) - w(\xi'-h)\| \leq \lambda_0 |\xi - \xi|,\quad \xi,\xi' \in [-h, -h+\delta_*).
\end{equation}
Define
\begin{equation}\label{lem:tilde_phi:delta}
\nu = \min\{1, \varepsilon_*, \varepsilon_* / \lambda_0, \varepsilon_* / \lambda^*, \nu_*\}.
\end{equation}
For proving the lemma, it suffices to establish the inequality
\begin{equation}\label{lem:tilde_phi:continuity}
|\tilde{\varphi}(t,x) - \tilde{\varphi}(t',x')| \leq \varepsilon,\quad (t',x') \in O_* :=
O_\nu(t,x) \cap ([\tau,\tau + \delta_*] \times \mathbb R^n).
\end{equation}
Firstly, due to the choice of $\delta_*$, $\alpha_0$, and $\lambda_\varphi$, from (\ref{lem:tilde_phi:overline_w}), (\ref{lem:tilde_phi:overline_lambda}), the first inequality in (\ref{lem:tilde_phi:delta_s}), and (\ref{lem:tilde_phi:lambda_0})--(\ref{lem:tilde_phi:continuity}), we derive
\begin{equation*}
\begin{array}{c}
|\varphi(t',x',\kappa_{t'}(\cdot)) - \varphi(t',w^*(-0),w^*(\cdot))|
\leq \lambda_\varphi\big(4 \varepsilon_* + \|x' - x\| + \|\kappa_{t'}(\cdot) - \kappa_t(\cdot)\|_1 \\[0.2cm]
 + \|\kappa_{t'}(-h) - \kappa_{t}(-h)\| + \|w^*(i h - t) - w^*(i h - t')\|\big) \leq 8 \lambda_\varphi \varepsilon_*
\end{array}
\end{equation*}
for any $(t',x') \in O_*$. Applying this estimate for $(t,x)$ and arbitrary $(t',x') \in O_*$, taking (\ref{def:tilde_phi}) into account, we obtain
\begin{equation*}
|\tilde{\varphi}(t,x) - \tilde{\varphi}(t',x')| \leq
|\varphi(t,w^*(-0),w^*(\cdot)) - \varphi(t',w^*(-0),w^*(\cdot))| + 16 \lambda_\varphi \varepsilon_*.
\end{equation*}
In accordance with the choice of $\varepsilon_*$, the second inequality in (\ref{lem:tilde_phi:delta_s}), and (\ref{lem:tilde_phi:delta}), from this inequality we conclude (\ref{lem:tilde_phi:continuity}).
\end{proof}

\begin{lemma}\label{lem:viscosity_and_D_solutions}
If the functional $\varphi\colon \mathbb G \mapsto \mathbb R$ is a viscosity solution of problem {\rm (\ref{Hamilton-Jacobi_equation}), (\ref{terminal_condition})}, then it satisfies {\rm (\ref{subdifferential_viscosity_inequality})} and  {\rm (\ref{superdifferential_viscosity_inequality})}.
\end{lemma}
\begin{proof}
Let $(\tau,z,w(\cdot)) \in \mathbb G_*$ and $(p_0,p) \in D^-\varphi(\tau,z,w(\cdot))$. Choose the number $\delta_*$ in accordance with Lemma~\ref{lem:delta_s}. Taking the function $\tilde{\varphi}$ by (\ref{def:tilde_phi}), define consistently the function $\tilde{\varphi}_*$ as follows
\begin{equation*}
\tilde{\varphi}_*(t,x) = \min\{0,\tilde{\varphi}(t,x) - \tilde{\varphi}(\tau,z) - (t - \tau) p_0 - \langle x - z, p \rangle\}
\end{equation*}
for $(t, x) \in [\tau,\tau + \delta_*] \times \mathbb R^n$, $\tilde{\varphi}_*(t,x) = \tilde{\varphi}_*(\tau + \delta_*,x)$ for $(t, x) \in (\tau + \delta_*,+\infty) \times \mathbb R^n$, and $\tilde{\varphi}_*(t,x) = \tilde{\varphi}_*(\tau - (t - \tau),x)$ for $(t, x) \in (-\infty,\tau) \times \mathbb R^n$. Then, according to Lemma~\ref{lem:tilde_phi}, this function is continuous on $\mathbb R \times \mathbb R^n$. The further proof of the fact that $\varphi$ satisfies (\ref{subdifferential_viscosity_inequality}) can be carried out in the same way as in Lemma 8.1 from \cite{Plaksin_2021}. In a similar way one can prove that (\ref{super_viscosity_solution}) implies (\ref{superdifferential_viscosity_inequality}).
\end{proof}

\begin{lemma}\label{lem:positive_part_deriv}
Let $\varphi \colon \mathbb G \mapsto \mathbb R$ satisfy $(\varphi_2)$ and $(\tau,z,w(\cdot)) \in \mathbb G_*$. Let $L \subset \mathbb R^n$ be a nonempty compact set. Suppose that
\begin{equation}\label{lem:positive_part_deriv:cond}
\partial^-_{1,l} \varphi(\tau,z,w(\cdot)) > 0,\quad l \in L.
\end{equation}
Then, there exist $\epsilon_\circ, \delta_\circ > 0$ such that
\begin{equation}\label{lem:positive_part_deriv:state}
\varphi(t, z+l(t-\tau), \kappa_t(\cdot)) - \varphi(\tau,z,w(\cdot)) > \epsilon_\circ (\tau - t)
\end{equation}
for any $t \in (\tau, \tau + \delta_\circ]$ and $l \in L$, where $\kappa(\cdot) \in \Lambda_0(\tau,z,w(\cdot))$.
\end{lemma}
\begin{proof}
The proof of the lemma is based on definition (\ref{lower_directional_derivatives}) of lower directional derivatives and condition $(\varphi_2)$.
\end{proof}

\begin{lemma}\label{lem:mvi}
Let $\varphi \colon \mathbb G \mapsto \mathbb R$ satisfy condition $(\varphi_2)$. Let $(\tau,z,w(\cdot)) \in \mathbb G_*$. Let $L \subset \mathbb R^n$ be a nonempty convex compact set. Suppose that {\em (\ref{lem:positive_part_deriv:cond})} holds. Then, for every $\delta > 0$, there exist
\begin{equation}\label{lem:mvi:t_s_y_s_p0_p}
\begin{array}{c}
(t_*,x_*) \in \Omega_\delta :=
\big\{(t,x) \in [\tau, \tau + \delta] \times \mathbb R^n \colon \min\limits_{l \in L} \|x - z - l (t - \tau)\| \leq \delta\big\},\\[0.2cm] (p_0,p) \in D^-\varphi(t_*,x_*,\kappa_{t_*}(\cdot)),\quad \kappa(\cdot) \in \Lambda_0(\tau,z,w(\cdot)),
\end{array}
\end{equation}
such that
\begin{equation}\label{lem:mvi:state}
p_0 + \langle l, p\rangle > 0,\quad l \in L.
\end{equation}
\end{lemma}
\begin{proof}
Take $\delta_* > 0$ from Lemma \ref{lem:delta_s}. According to (\ref{lem:mvi:t_s_y_s_p0_p}), one can take $\alpha_\Omega > 0$ such that $\{\|x\|,\|w(\cdot)\|_\infty\} \leq \alpha_\Omega$ for any $(t,x) \in \Omega_{\delta_*}$. Due to condition $(\varphi_2)$, define $\lambda_\varphi = \lambda_\varphi(\alpha_\Omega)$. Then, for every $(t,x), (t_*,x_*) \in \Omega_{\delta_*}$ satisfying $t \geq t_*$, and $\chi(\cdot) \in \Lambda_0(t_*,x_*,\kappa_{t_*}(\cdot))$, we derive
\begin{equation}\label{lem:mvi:lambda_phi}
|\varphi(t,x,\chi_t(\cdot)) - \varphi(t,x,\kappa_t(\cdot))|
\leq \lambda_\varphi \|\chi_t(\cdot) - \kappa_t(\cdot)\|_1
\leq \lambda_\varphi \|x_* - z\| (t - t_*).
\end{equation}
According to Lemma \ref{lem:positive_part_deriv}, there exist $\epsilon_\circ,\delta_\circ > 0$ such that (\ref{lem:positive_part_deriv:state}) holds. Then, without loss of generality, we can suppose that
\begin{equation}\label{lem:mvi:nu}
\delta \leq \min\{\delta_*,\delta_\circ\},\quad \delta < \epsilon_\circ / (\lambda_\varphi (1 + \lambda_L)),\quad \lambda_L = \max\{\|l\|\,|\,l \in L\}.
\end{equation}

For each $k \in \mathbb N$, define the function
\begin{equation}\label{lem:mvi:gamma_k}
\gamma_k(t,x,\xi,y) = \varphi(t,x,\kappa_t(\cdot)) + k \|x - y\|^2 + k (t - \xi)^2 - \epsilon_\circ (\xi - \tau),
\end{equation}
where $(t,x) \in \Omega_\delta$ and $(\xi,y) \in \Omega_\delta^* := \{(\xi,y) \in\Omega_\delta \colon \min\limits_{l \in L} \|y - z - l (\xi - \tau)\| = 0\}$. According to the choice of $\delta$ and Lemma \ref{lem:tilde_phi}, $\gamma_k$ is continuous. The set $\Omega_\delta \times \Omega_\delta^*$ is compact. Therefore, there exists $(t_k,x_k,\xi_k,y_k) \in \Omega_\delta \times \Omega_\delta^*$ such that
\begin{equation}\label{lem:mvi:min_point}
\gamma_k(t_k,x_k,\xi_k,y_k) = \min\limits_{(t,x,\xi,y) \in \Omega_\delta \times \Omega_\delta^*} \gamma_k(t,x,\xi,y).
\end{equation}
Furthermore, without loss of generality, we suppose that $(t_k,x_k,\xi_k,y_k) \to (\overline{t},\overline{x},\overline{\xi},\overline{y}) \in \Omega_\delta \times \Omega_\delta^*$ as $k \to \infty$. Due to (\ref{lem:mvi:min_point}), we have
\begin{equation}\label{lem:mvi:min_and_init_point}
\gamma_k(t_k,x_k,\xi_k,y_k) \leq \gamma_k(\tau,z,\tau,z) = \varphi(\tau,z,w(\cdot)).
\end{equation}
Hence, we obtain
\begin{equation}\label{lem:mvi:limit}
\overline{t} = \overline{\xi},\quad \overline{x} = \overline{y}.
\end{equation}

Let us show that $\overline{t} < \tau + \delta$. For the sake of a contradiction, suppose that $\overline{t} = \tau + \delta$. Then, applying Lemma \ref{lem:tilde_phi} and (\ref{lem:positive_part_deriv:state}), (\ref{lem:mvi:gamma_k}),  we derive
\begin{equation*}
\begin{array}{c}
\liminf\limits_{k \to \infty} \gamma_k(t_k,x_k,\xi_k,y_k) \geq \lim\limits_{k \to \infty}\big(\varphi(t_k,x_k,\kappa_{t_k}(\cdot)) - \epsilon_\circ (\xi_k - \tau) \big) \\[0.2cm]
 = \varphi(\tau + \delta, \overline{y}, \kappa_{\tau + \delta}(\cdot)) - \epsilon_\circ \delta > \varphi(\tau,z,w(\cdot)).
 \end{array}
\end{equation*}
This inequality contradicts (\ref{lem:mvi:min_and_init_point}).

In accordance with $\overline{t} < \tau + \delta$ and (\ref{lem:mvi:limit}), one can take $k \in \mathbb N$ so that
\begin{equation}\label{lem:mvi:k}
t_k < \tau + \delta,\quad \xi_k < \tau + \delta,\quad \|x_k - y_k\| \leq \delta / 4,\quad \lambda_L |t_k - \xi_k| \leq \delta / 4,
\end{equation}
where the number $\lambda_L$ is defined in (\ref{lem:mvi:nu}). Put
\begin{equation}\label{lem:mvi:p0_p}
p_0 = - 2 k (t_k - \xi_k) - \lambda_\varphi \|x_k - z\|,\quad p = - 2 k (x_k - y_k).
\end{equation}

Let us prove the inclusion $(p_0,p) \in D^- \varphi (t_k,x_k,\kappa_{t_k}(\cdot))$. Since $(\xi_k,y_k) \in \Omega_\delta^*$, there exists $l_k$ such that $y_k = z + l_k (\xi_k - \tau)$. Then, due to definition $\lambda_L$ in (\ref{lem:mvi:nu}) and (\ref{lem:mvi:k}), we have
\begin{equation*}
\|x - z - l_k (t - \tau)\| \leq \|x - x_k\| + \|x_k - y_k\| + \lambda_L |t - t_k| + \lambda_L |t_k - \xi_k| \leq \delta
\end{equation*}
for any $(t,x) \in O^+_{\delta (1 + \lambda_L) / 4}(t_k,x_k)$. It means that $(t,x) \in \Omega_\delta$ for any $(t,x) \in O^+_{\delta (1 + \lambda_L) / 4}(t_k,x_k)$. Applying (\ref{lem:mvi:gamma_k}), (\ref{lem:mvi:min_point}), we obtain
\begin{equation*}
\begin{array}{c}
0 \leq \gamma_k(t, x, \xi_k, y_k) -  \gamma_k(t_k, x_k, \xi_k, y_k)
= \varphi(t,x,\kappa_t(\cdot)) - \varphi(t_k,x_k,\kappa_{t_k}(\cdot))  \\[0.2cm]
+ k \|x - x_k\|^2 + 2 k \langle x - x_k, x_k - y_k \rangle + k (t - t_k)^2 + 2 k (t - t_k) (t_k - \xi_k)
\end{array}
\end{equation*}
for any $(t,x) \in O^+_{\delta (1 + \lambda_L) / 4}(t_k,x_k)$. Then, taking into account (\ref{subdifferential}), (\ref{lem:mvi:lambda_phi}) for $(t_*,x_*) = (t_k,x_k)$, and (\ref{lem:mvi:p0_p}), we conclude $(p_0,p) \in D^-\varphi(t_k,x_k,\kappa_{t_k}(\cdot))$.

Let us prove (\ref{lem:mvi:state}). Let $l \in L$. Since $L$ is convex, then we have
\begin{equation*}
l_\nu = (l \nu + l_k (\xi_k - \tau))/ (\nu + \xi_k - \tau) \in L,\quad \nu \in (0, \delta + \tau - \xi_k).
\end{equation*}
From this inclusion and $(\xi_k,y_k) \in \Omega_\delta^*$, we derive $\|y_k+l \nu - z - l_\nu (\xi_k+\nu - \tau)\| = 0$ that means the inclusion $(\xi_k+\nu,y_k+l \nu) \in \Omega^*_\delta$, $\nu \in (0, \delta + \tau - \xi_*)$. Then, according to (\ref{lem:mvi:gamma_k}), (\ref{lem:mvi:min_point}), for every $\nu \in (0,\delta+ \tau - \xi_*)$, we obtain
\begin{equation*}
\begin{array}{c}
0 \leq \gamma_k(t_k,x_k,\xi_k+\nu,y_k+l \nu,) - \gamma_k(t_k,x_k,\xi_k,y_k)   \\[0.2cm]
= k \|l\|^2 \nu^2 - 2 k \langle l, x_k - y_k \rangle \nu + k \nu^2 - 2 k (t_k - \xi_k) \nu - \epsilon_\circ \nu.
\end{array}
\end{equation*}
Dividing this inequality by $\nu$ and passing to the limit as $\nu \to + 0$, we get
\begin{equation}\label{lem:mvi:xi_var}
\epsilon_\circ \leq - 2 k \langle x_k - y_k, l \rangle - 2 k (t_k - \xi_k).
\end{equation}
Since $(t_k,x_k) \in \Omega_\delta$, there exists $l_x \in L$ such that $\|x_k - z - l_x(t_k - \tau)\| \leq \delta$. Then, using (\ref{lem:mvi:nu}), we derive
\begin{equation}\label{lem:mvi:x_k_z}
\|x_k - z\| \leq \|x_k - z - l_x(t_k - \tau)\| + \|l_x\| (t_k - \tau) \leq (1 + \lambda_L) \delta < \epsilon_\circ / \lambda_\varphi .
\end{equation}
From (\ref{lem:mvi:p0_p})--(\ref{lem:mvi:x_k_z}), we conclude (\ref{lem:mvi:state}). Thus, taking $(t_*,x_*) = (t_k,x_k)$ we obtain statement of the lemma.
\end{proof}

\begin{lemma}\label{lem:directional_derivative_continuous}
Let $\varphi \colon \mathbb G \mapsto \mathbb R$ satisfy $(\varphi_2)$ and $(\tau,z,w(\cdot)) \in \mathbb G_*$. If there exists $l_* \in \mathbb R^n$ such that $\partial^-_{1,l_*}\varphi(t,z,w(\cdot)) \in \mathbb R$, then $\partial^-_{1,l}\varphi(t,z,w(\cdot)) \in \mathbb R$ for every $l \in \mathbb R^n$, and the function $\phi(l) = \partial^-_{1,l}\varphi(t,z,w(\cdot))$, $l \in \mathbb R^n$, is continuous. If there exists $l_* \in \mathbb R^n$ such that $\partial^-_{1,l_*}\varphi(t,z,w(\cdot)) = + \infty$, then $\partial^-_{1,l}\varphi(t,z,w(\cdot))=  + \infty$ for every $l \in \mathbb R^n$.
\end{lemma}
\begin{proof}
follows directly from condition $(\varphi_2)$ and definition (\ref{directional_derivatives}).
\end{proof}

\begin{lemma}\label{lem:D_and_dini_solutions}
If the functional $\varphi\colon \mathbb G \mapsto \mathbb R$ satisfies {\em (\ref{subdifferential_viscosity_inequality})} and {\em (\ref{superdifferential_viscosity_inequality})}, then it satisfies {\em (\ref{lower_directional_derivative_inequality})} and {\em (\ref{upper_directional_derivative_inequality})}.
\end{lemma}
\begin{proof}
Let us prove (\ref{lower_directional_derivative_inequality}). For the sake of a contradiction, suppose that there exist $(\tau,z,w(\cdot)) \in \mathbb G_*$ and $s \in \mathbb R^n$ such that
\begin{equation*}
\partial^-_{1,l} \varphi(\tau,z,w(\cdot)) + \langle \partial^{ci}_{\tau,w} g(\tau,w(\cdot)), s \rangle + H(\tau,z,w(-h),s) - \langle l,s \rangle > 0
\end{equation*}
for any $l \in F(z,w(-h)) + \partial^{ci}_{\tau,w} g(\tau,w(\cdot))$. In the case if there exists $l_* \in F(z,w(-h)) + \partial^{ci}_{\tau,w} g(\tau,w(\cdot))$ such that $\partial^-_{1,l_*} \varphi(\tau,z,w(\cdot)) \in \mathbb R$, then, taking into account Lemma \ref{lem:directional_derivative_continuous}, one can take $\eta, \varepsilon > 0$ so that
\begin{equation}\label{lem:f_d:eta_epsilon}
\partial^-_{1,l} \varphi(\tau,z,w(\cdot)) + \langle \partial^{ci}_{\tau,w} g(\tau,w(\cdot)), s \rangle
+ H(\tau,z,w(-h),s) - \langle l,s \rangle > \varepsilon
\end{equation}
for any $l \in F^\eta(z,w(-h)) + \partial^{ci}_{\tau,w} g(\tau,w(\cdot))$.
If there exists $l_0 \in F(z,w(-h))$ such that $\partial^-_{1,l_0} \varphi(t,z,w(\cdot)) = + \infty$, then, in accordance with Lemma \ref{lem:directional_derivative_continuous}, inequality (\ref{lem:f_d:eta_epsilon}) also holds.

Put $L = F^\eta(z,w(-h)) + \partial^{ci}_{\tau,w} g(\tau,w(\cdot))$. According to conditions $(H_1)$ and (\ref{def:F}), there exists $\delta_2 > 0$ such that
\begin{equation}\label{lem:f_d:delta_2}
\begin{array}{c}
|H(t,x,w(t-\tau-h),s) - H(\tau,z,w(-h),s)| \leq \varepsilon,\\[0.2cm]
F^0(x,w(t-\tau-h)) \subset F^\eta(z,w(-h)),
\end{array}
\quad (t,x) \in \Omega_{\delta_2}.
\end{equation}
Define the functional $\varphi_* \colon \mathbb G \mapsto \mathbb R$ by
\begin{equation*}
\begin{array}{c}
\varphi_*(t,x,r(\cdot)) = \varphi(t,x,r(\cdot))  \\[0.2cm]
+ \big(\langle \partial^{ci}_{\tau,w} g(\tau,w(\cdot)), s \rangle + H(\tau,z,w(-h),s) - \varepsilon\big) (t - \tau) - \langle x, s \rangle,
\end{array}
\ \ (t,x,r(\cdot)) \in \mathbb G.
\end{equation*}
Since $\varphi$ satisfy $(\varphi_2)$, then $\varphi_*$ satisfy $(\varphi_2)$. Moreover, from (\ref{lem:f_d:eta_epsilon}), we derive $\partial^-_{1,l} \varphi_*(\tau,z,w(\cdot)) > 0$, $l \in L$. Applying Lemma \ref{lem:mvi} to the functional $\varphi_*$, the set $L$, and $\delta = \min\{\delta_1,\delta_2\}$, we obtain that there exist $(t_*,x_*) \in \Omega_\delta$ and $(p_0,p) \in D^-\varphi_*(t_*,x_*,\kappa_{t_*}(\cdot))$ such that
\begin{equation}\label{lem:f_d:p0_p}
p_0 + \langle l,p \rangle > 0,\quad l \in L = F^\eta(z,w(-h)) + \partial^{ci}_{\tau,w} g(\tau,w(\cdot)).
\end{equation}
Let us define
\begin{equation*}
p'_0 = p_0 - H(\tau,z,w(-h),s) - \langle \partial^{ci}_{\tau,w} g(\tau,w(\cdot)), s \rangle + \varepsilon,\quad p' = p + s.
\end{equation*}
Then, we have $(p'_0,p') \in D^-\varphi(t_*,x_*,\kappa_{t_*}(\cdot))$. Thus, taking the choice of $\delta$ into account, from (\ref{def:F}), (\ref{subdifferential_viscosity_inequality}), (\ref{lem:f_d:delta_2}), and condition $(H_3)$, we obtain
\begin{equation*}
\begin{array}{c}
0 \geq p'_0 + \langle \partial^{ci}_{\tau,w} g(\tau,w(\cdot)), p' \rangle +
H(t_*,x_*,w(t_*-\tau-h),p') \\[0.3cm]
\geq p_0 + \langle \partial^{ci}_{\tau,w} g(\tau,w(\cdot)), p \rangle - c_H \big(1 + \|x_*\| + \|w(t_*-\tau-h)\|\big) \|p\| \\[0.3cm]
= p_0 + \langle \partial^{ci}_{\tau,w} g(\tau,w(\cdot)), p \rangle + \min\limits_{l \in F^0(x_*,w(t_*-\tau-h))} \langle l, p \rangle \\[0.3cm]
\geq p_0 + \langle \partial^{ci}_{\tau,w} g(\tau,w(\cdot)), p \rangle + \min\limits_{l \in F^\eta(z,w(-h))} \langle l, p \rangle.
\end{array}
\end{equation*}
This estimate contradicts (\ref{lem:f_d:p0_p}). Thus, (\ref{lower_directional_derivative_inequality}) holds. For $\partial^+_l \varphi(t,z,w(\cdot))$ and $D^+ \varphi(t,z,w(\cdot))$, one can establish statements similar to Lemmas \ref{lem:mvi}, \ref{lem:directional_derivative_continuous} and prove (\ref{upper_directional_derivative_inequality}).
\end{proof}

\begin{lemma}\label{lem:dini_and_C1_solutions}
If the functional $\varphi \colon \mathbb G \mapsto \mathbb R$ satisfies $(\varphi_1)$, $(\varphi_2)$, {\rm (\ref{lower_directional_derivative_inequality})}, and {\rm (\ref{upper_directional_derivative_inequality})}, then it is $\mathrm{C^1}$-minimax solution of problem {\rm (\ref{Hamilton-Jacobi_equation}), (\ref{terminal_condition})}.
\end{lemma}
\begin{proof}
Let us prove (\ref{def:upper_minmax_solution}) for any $i \in \overline{0,I-1}$, $(\tau,w(\cdot)) \in [i h, (i + 1) h) \times \mathrm{C}^1$, $t \in (\tau,(i+1)h]$, $\eta \in (0,1]$, and $s \in \mathbb R^n$.

For the sake of a contradiction, suppose that there exist $i \in \overline{0,I-1}$, $(\tau,w(\cdot)) \in [i h, (i + 1) h) \times \mathrm{C}^1$, $t \in (\tau,(i+1)h]$, $\eta  \in (0,1]$, $s \in \mathbb R^n$, and $\varepsilon > 0$ such that
\begin{equation}\label{lem:d_c:contr}
\varphi(t,x(t),x_t(\cdot)) - \varphi(\tau,w(-0),w(\cdot)) + \omega(\tau,t,x(\cdot),s) > \varepsilon
\end{equation}
for any $x(\cdot) \in X^\eta := X^\eta(\tau,w(-0),w(\cdot))$. Let
\begin{equation*}
\xi_* = \max \Big\{\xi \in [\tau, t] \colon \min\limits_{x(\cdot) \in X^\eta}  \big(\varphi(\xi,x(\xi),x_\xi(\cdot)) + \omega(\xi,\tau,x(\cdot),s)\big) \leq \beta(\xi)\Big\},
\end{equation*}
where $\beta(\xi) = \varphi(\tau,w(-0),w(\cdot)) + \varepsilon (\xi - \tau) / (t - \tau)$, $\xi \in [\tau, t]$. Similar to Assertion~2 from \cite{Plaksin_2021b}, one can show that the set $X^\eta$ is a compact on the space of continuous functions and there exist $\alpha_X, \lambda_X > 0$ such that
\begin{equation*}
\|x(\xi)\| \leq \alpha_X,\quad \|x(\xi) - x(\xi')\| \leq \lambda_X |\xi - \xi'|,\quad \xi,\xi' \in [\tau-h,\vartheta],\quad x(\cdot) \in X^\eta.
\end{equation*}
Using conditions $(\varphi_1)$ and $(\varphi_2)$, one can state the functional $\hat{\varphi}(\tau,w(\cdot)) = \varphi(\tau,w(-0),w(\cdot))$ is continuous at any $(\tau,w(\cdot)) \in [0,\vartheta] \times \mathrm{Lip}$. Thus, taking into account  conditions $(g)$, $(H_1)$, and definition (\ref{def:omega}) of $\omega$, we can establish that these maximum and minimum are archived. In accordance with (\ref{lem:d_c:contr}), we have $\xi_* < t$. Let the function $x(\cdot) \in X^\eta$ be such that
\begin{equation}\label{lem:d_c:xi_s}
\varphi(\xi_*,x(\xi_*),x_{\xi_*}(\cdot)) + \omega(\tau,\xi_*,x(\cdot),s) \leq \beta(\xi_*).
\end{equation}
Denote $\varepsilon_* = \varepsilon / (5 (t - \tau))$. Due to (\ref{lower_directional_derivative_inequality}), there exists $l_* \in F^0(x(\xi_*),x(\xi_*-h)) + \partial^{ci}_{\tau,w} g(\xi_*,x_{\xi_*}(\cdot))$ satisfying
\begin{equation}
\begin{array}{rcl}
\partial^-_{1,l_*} \varphi(\xi_*,x(\xi_*),x_{\xi_*}(\cdot)) & + & \langle \partial^{ci}_{\tau,w} g(\xi_*,x_{\xi_*}(\cdot)) ,s \rangle\\[0.2cm]
& + & H(\xi_*,x(\xi_*),x(\xi_*-h),s) - \langle l_*,s \rangle \leq \varepsilon_*.
\end{array}
\end{equation}
Redefine $x(\cdot)$ on the interval $[\xi_*,\vartheta]$ so that $\mathrm{d} x(\xi) / \mathrm{d} \xi = l_*$, $\xi \in [\xi_*, \vartheta]$. According to condition $(\varphi_2)$, define $\lambda_\varphi = \lambda_\varphi(\alpha_X)$. Then, due to (\ref{def:F}), (\ref{lower_directional_derivatives}), condition $(H_1)$, and Lemma \ref{lem:dg} in which, since $w(\cdot) \in \mathrm{C}^1$, we can take $\delta_w \in (\xi_* - \tau, t - \tau)$, there exist $t_* \in (\xi_*,\min\{\tau + \delta_w, \xi_* +  \varepsilon/(\lambda_\varphi \lambda_X)]$ such that
\begin{equation}
\displaystyle\frac{\varphi(t_*,x(t_*),\kappa_{t_*}(\cdot)) - \varphi(\xi_*,x(\xi_*),x_{\xi_*}(\cdot))}{t_* - \xi_*} \leq \partial^-_{1,l_*} \varphi(\xi_*,x(\xi_*),x_{\xi_*}(\cdot)) + \varepsilon_*,
\end{equation}
where $\kappa(\cdot) \in \Lambda_0(\xi_*,x(\xi_*),x_{\xi_*}(\cdot))$, and
\begin{equation}\label{lem:d_c:t_s}
\begin{array}{c}
F^0(x(\xi_*),x(\xi_*-h)) \subset F^{\eta/2}(x(\xi),x(\xi-h)),\\[0.3cm]
\big|H(\xi,x(\xi),x_{\xi}(\cdot),s) - H(\xi_*,x(\xi_*),x_{\xi_*}(\cdot),s)\big| \leq \varepsilon_*,\\[0.3cm]
\displaystyle\bigg|\frac{\mathrm{d}}{\mathrm{d} \xi}\Big(g(\xi,x_\xi(\cdot))\Big) - \partial^{ci}_{\tau,w} g(\xi_*,x_{\xi_*}(\cdot))\bigg| \leq \min\bigg\{\frac{\eta}{2}, \frac{\varepsilon_*}{\|s\| + 1}\bigg\}
\end{array}
\end{equation}
for any $\xi \in [\xi_*,t_*]$. Then, we have
\begin{equation*}
\frac{\mathrm{d}}{\mathrm{d} \xi} \Big(x(\xi) - g(\xi,x(\xi-h))\Big) \in F^\eta(x(\xi),x(\xi-h))\quad {a.e. }\ \xi \in [\xi_*,t_*].
\end{equation*}
Redefining the function $x(\cdot)$ on the interval $[t_*,\vartheta]$ such that $x(\cdot) \in X^\eta$, according to the choice of $\alpha_X$, $\lambda_X$, $\lambda_\varphi$, and $t_*$, we derive
\begin{equation*}
\big|\varphi(t_*,x(t_*),\kappa_{t_*}(\cdot)) - \varphi(t_*,x(t_*),x_{t_*}(\cdot))
\leq \lambda_\varphi \lambda_X (t_* - \xi_*)^2 \leq \varepsilon_* (t_* - \xi_*).
\end{equation*}
Using this estimate, from (\ref{lem:d_c:xi_s})--(\ref{lem:d_c:t_s}), taking the definition of $\varepsilon_*$ and $\beta$ into account, we obtain
\begin{equation*}
\varphi(t_*,x(t_*),x_{t_*}(\cdot)) + \omega(\tau,t_*,x(\cdot),s) \leq \beta(t_*).
\end{equation*}
This inequality contradict the choice of $\xi_*$.
\end{proof}

\subsection{Proof of Theorem \ref{teo:application}}

\begin{lemma}\label{lem:value_functional}
The value functional $\hat{\varphi}$ defined in {\rm (\ref{def:value_functional})} is the functional from Lemma~\ref{lem:ex_un_minimax_sol}.
\end{lemma}
\begin{proof}
Let us prove (\ref{def:upper_minmax_solution}). Let $(\tau,w(\cdot)) \in \mathbb [0,\vartheta) \times \mathrm{Lip}$, $t \in (\tau,\vartheta]$, $s \in \mathbb R^n$, and $\varepsilon > 0$. Due to definition (\ref{def:value_functional}) of $\hat{\varphi}$, one can show (see, e.g., \cite[Chapter VIII, Theorem 1.9]{Bardi_Capuzzo-Dolcetta_1997}) that there exists $u(\cdot) \colon [\tau,\vartheta] \mapsto \mathbb U$ such that
\begin{equation*}
\hat{\varphi}(t,x_t(\cdot)) + \int_\tau^t f^0(\xi,x(\xi),x(\xi-h),u(\xi)) \mathrm{d} \xi \leq \hat{\varphi}(\tau,x_\tau(\cdot))
\end{equation*}
where $x(\cdot) \in \Lambda(\tau,w(-0),w(\cdot))$ is the Lipschitz continuous function  satisfying neutral-type equation (\ref{def:neutral_type_equation}). From this estimate, using (\ref{def:hamiltonian}), we get (\ref{def:upper_minmax_solution}).

Let us prove (\ref{def:lower_minmax_solution}). Let $(\tau,w(\cdot)) \in \mathbb [0,\vartheta) \times \mathrm{Lip}$, $t \in (\tau,\vartheta]$, $s \in \mathbb R^n$, and $\varepsilon > 0$. Due to Assertion 2 from \cite{Plaksin_2021b}, there exist $\alpha_X, \lambda_X > 0$ such that
\begin{equation*}
\|x(\xi)\| \leq \alpha_X,\quad \|x(\xi) - x(\xi')\| \leq \lambda_X |\xi - \xi'|,\quad \xi,\xi' \in [\tau-h,\vartheta]
\end{equation*}
for any Lipschitz continuous function $x(\cdot) \in \Lambda(\tau,w(-0),w(\cdot))$ satisfying  neutral-type equation (\ref{def:neutral_type_equation}). Then, taking into account condition $(H_1)$ and continuity of the functions $f$ and $f^0$, there exist $\delta > 0$ such that
\begin{equation*}
\begin{array}{c}
\big|H(\xi,x(\xi),x(\xi-h),s) - H(\xi',x(\xi'),x(\xi'-h),s)\big| \leq \varepsilon,\\[0.2cm]
\big|f(\xi,x(\xi),x(\xi-h),u) - f(\xi',x(\xi'),x(\xi'-h),u)\big| \leq \varepsilon,\\[0.2cm]
\big|f^0(\xi,x(\xi),x(\xi-h),u) - f^0(\xi',x(\xi'),x(\xi'-h),u)\big| \leq \varepsilon
\end{array}
\end{equation*}
for any $\xi,\xi' \in [\tau,\vartheta]$ satisfying $|\xi - \xi'| \leq \delta$, any $u \in \mathbb U$, and any Lipschitz continuous $x(\cdot) \in \Lambda(\tau,w(-0),w(\cdot))$ satisfying  (\ref{def:neutral_type_equation}). Let $k \in \mathbb N$ be such that $\Delta t := (t - \tau)/k \leq \delta$ and $t_i = \tau + i \Delta t$, $i \in \overline{0,k}$. Let the function $u(\cdot) \colon [\tau,\vartheta] \mapsto \mathbb U$ and the Lipschitz continuous function $x(\cdot)$ satisfy (\ref{def:neutral_type_equation}) and the following feedback rule:
\begin{equation*}
u(\xi) = u_i,\quad \xi \in [t_i,t_{i+1}),\quad i \in \overline{0,k-1}
\end{equation*}
where, in accordance with (\ref{def:hamiltonian}), the value $u_i \in \mathbb U$ can be taken satisfying $H(t_i,x(t_i),x(t_i-h),s) = \langle f(t_i,x(t_i),x(t_i-h),u_i), s \rangle + f^0(t_i,x(t_i),x(t_i-h),u_i)$. Then, taking into account (\ref{def:omega}) and (\ref{def:value_functional}), we derive
\begin{equation*}
\begin{array}{c}
\hat{\varphi}(t_{i+1},x_{t_{i+1}}(\cdot)) + \omega(t_i,t_{i+1},x(\cdot),s) + 3 (t_{i+1} - t_i) \varepsilon\\[0.2cm]
\geq
\displaystyle\hat{\varphi}(t_{i+1},x_{t_{i+1}}(\cdot)) + \int_{t_i}^{t_{i+1}} f^0(\xi,x(\xi),x(\xi-h),u(\xi)) \mathrm{d} \xi \geq \hat{\varphi}(t_{i},x_{t_{i}}(\cdot))
\end{array}
\end{equation*}
for any $i \in \overline{0,k-1}$. Using this estimate for each $i \in \overline{0,k-1}$ and $\varepsilon > 0$, we conclude (\ref{def:lower_minmax_solution}).
\end{proof}

%\begin{acknowledgements}
%\end{acknowledgements}



\begin{thebibliography}{}

\bibitem{Aubin_Haddad_2002}
Aubin, J.P., Haddad, G.:
History path dependent optimal control and portfolio valuation and management.
Positivity. 6(3), 331--358 (2002). https://doi.org/10.1023/A:1020244921138

\bibitem{Bayraktar_Keller_2018}
Bayraktar, E., Keller, C.:
Path-dependent Hamilton--Jacobi equations in infinite dimensions.
Journal of Functional Analysis. 275(8), 2096--2161 (2018).
https://doi.org/10.1016/j.jfa.2018.07.010

\bibitem{Bardi_Capuzzo-Dolcetta_1997}
Bardi, M., Capuzzo-Dolcetta, I.:
Optimal Control and Viscosity Solutions of Hamilton-Jacobi-Bellman Equations,
Birkh\"auser, Boston (1997)

\bibitem{Clarke_Ledyaev_1994}
Clarke, F.H., Ledyaev, Yu.S.:
Mean Value Inequalities in Hilbert Space.
Transactions of the American Mathematical Society. 344(1), 307--324 (1994).
https://doi.org/10.2307/2154718

\bibitem{Crandall_Lions_1983}
Crandall, M.G., Lions, P.-L.:
Viscosity solutions of Hamilton-Jacobi equations.
Transactions of the American Mathematical Society. 277(1), 1--42 (1983).
https://doi.org/10.2307/1999343

\bibitem{Crandall_Evans_Lions_1984}
Crandall, M.G., Evans, L.C., Lions, P.-L.:
Some properties of viscosity solutions of Hamilton-Jacobi equations.
Transactions of the American Mathematical Society. 282(2), 487--502 (1984).
https://doi.org/10.2307/1999247

\bibitem{Dupire_2009}
Dupire, B.:
Functional {It\^{o}} calculus. https://ssrn.com/abstract=1435551 (2009).
Accessed 25 July 2009

\bibitem{Ekren_Touzi_Zhang_2016}
Ekren, I., Touzi, N., Zhang, J.:
Viscosity solutions of fully nonlinear parabolic path dependent PDEs: Part I.
Annals of Probability. 44(2), 1212--1253 (2016).
https://doi.org/10.1214/15-AOP1027

\bibitem{Filippov_1988}
Filippov, A.F.:
Differential equations with discontinuous Righthand Sides.
Springer, Berlin (1988)

\bibitem{Gomoyunov_Lukoyanov_Plaksin_2021}
Gomoyunov, M.I., Lukoyanov, N.Yu., Plaksin, A.R.:
Path-dependent Hamilton-Jacobi equations: the minimax solutions revised.
Applied Mathematics and Optimization. 84(1), S1087--S1117 (2021)

\bibitem{Gomoyunov_Plaksin_2019}
Gomoyunov, M.I., Plaksin, A.R.:
On basic equation of differential games for neutral-type systems.
Mechanics of Solids. 54(2), 131--143 (2019)

\bibitem{Hale_1977}
Hale, J.:
Theory of Functional Differential Equations.
Springer-Verlag, New York (1977)

\bibitem{Kaise_2015}
Kaise, H.: Path-dependent differential games of inf-sup type and Isaacs partial differential equations.
Proceedings of the 54th IEEE Conference on Decision and Control (CDC). 1972--1977 (2015).
https://doi.org/10.1109/CDC.2015.7402496

\bibitem{Kaise_Kato_Takahashi_2018}
Kaise, H., Kato, T., Takahashi, Y.:
Hamilton-Jacobi partial differential equations with path-dependent terminal costs under superlinear Lagrangians.
Proceedings of the 23rd International Symposium on Mathematical Theory of Networks and Systems (MTNS). 692--699 (2018)

\bibitem{Kim_1999}
Kim, A.V.:
Functional Differential Equations: Application of $i$-Smooth Calculus.
Kluwer Academic Publishers, Dordrecht, The Netherlands (1999)

\bibitem{Krasovskii_Subbotin_1988}
Krasovskii, N.N., Subbotin, A.I.:
Game-Theoretical Control Problems,
Springer, New York (1988)

\bibitem{Krasovskii_Krasovskii_1995}
Krasovskii, A.N., Krasovskii, N.N.:
Control under Lack of Information,
Birkh\"auser, Berlin (1995)

\bibitem{Lukoyanov_2000}
Lukoyanov, N.Yu.:
A Hamilton-Jacobi type equation in control problems with hereditary information.
Journal of Applied Mathematics and Mechanics. 64, 243--253 (2000).
https://doi.org/10.1016/S0021-8928(00)00046-0

\bibitem{Lukoyanov_2003}
Lukoyanov, N.Yu.:
Functional Hamilton-Jacobi type equation in ci-derivatives for systems with distributed delays.
Nonlinear Functional Analysis and Applications. 8(3), 365--397 (2003)

\bibitem{Lukoyanov_2010a}
Lukoyanov, N.Yu.:
On optimality conditions for the guaranteed result in control problems for time-delay systems.
Proceedings of the Steklov Institute of Mathematics. 1, 175--187 (2010)

\bibitem{Lukoyanov_2010b}
Lukoyanov, N.Yu.:
Minimax and viscosity solutions in optimization problems for hereditary systems.
Proceedings of the Steklov Institute of Mathematics. 2, 214--225 (2010).
https://doi.org/10.1134/S0081543810060179

\bibitem{Lukoyanov_Plaksin_2020}
Lukoyanov, N.Yu., Plaksin, A.R.:
Hamilton-Jacobi Equations for Neutral-Type Systems: Inequalities for Directional Derivatives of Minimax Solutions.
Minimax Theory and its Applications. 5(2), 369--381 (2020)

\bibitem{Lukoyanov_Plaksin_2020b}
Lukoyanov, N.Yu., Plaksin, A.R.:
On the Theory of Positional Differential Games for Neutral-Type Systems.
Proceedings of the Steklov Institute of Mathematics. 309(1), S83--S92 (2020).https://doi.org/10.1134/S0081543820040100

\bibitem{Natanson_1960}
Natanson, I.P.:
Theory of Functions of a Real Variable. Volume 2.
Frederick Ungar Publishing Co., New-York (1960)

\bibitem{Pham_Zhang_2014}
Pham, T., Zhang, J.:
Two person zero-sum game in weak formulation and path dependent {Bellman--Isaacs} equation.
SIAM Journal on Control and Optimization. 52(4), 2090--2121 (2014).
https://doi.org/10.1137/120894907

\bibitem{Plaksin_2020}
Plaksin, A.:
Minimax and Viscosity Solutions of Hamilton-Jacobi-Bellman Equations for Time-Delay Systems.
Journal of Optimization Theory and Applications. 187(1), 22--42 (2020).
https://doi.org/10.1007/s10957-020-01742-6

\bibitem{Plaksin_2021}
Plaksin, A.:
Viscosity Solutions of Hamilton-Jacobi-Bellman-Isaacs Equations for Time-Delay Systems.
SIAM Journal on Control and Optimization. 59(3), 1951--1972 (2021).
https://doi.org/10.1137/20M1311880

\bibitem{Plaksin_2021b}
Plaksin, A.:
On the Minimax Solution of the Hamilton-Jacobi Equations for Neutral-Type Systems:
the Case of an Inhomogeneous Hamiltonian.
Differential Equations. 57(11), 1516--1526 (2021).
https://doi.org/10.1134/S0012266121110100

\bibitem{Soner_1988}
Soner, H.M.:
On the Hamilton-Jacobi-Bellman equations in Banach spaces.
Journal of Optimization Theory and Applications. 57(3), 429-437 (1988).
https://doi.org/10.1007/BF02346162

\bibitem{Subbotin_1980}
Subbotin, A. I.:
A generalization of the basic equation of the theory of differential games.
Soviet Mathematics - Doklady. 22, 358-362 (1980)

\bibitem{Subbotin_1984}
Subbotin, A.I.:
Generalization of the main equation of differential game theory.
Journal of Optimization Theory and Applications. 43(1), 151-162 (1984)

\bibitem{Subbotin_1993}
Subbotin, A.I.:
On a property of the subdifferential.
Mathematics of the USSR - Sbornik. 74(1), 63-78 (1993)

\bibitem{Subbotin_1995}
Subbotin, A.I.:
Generalized Solutions of First Order PDEs: The Dynamical Optimization Perspective.
Birkh\"{a}user, Boston (1995)

\bibitem{Zhou_2019}
Zhou, J.:
Delay optimal control and viscosity solutions to associated Hamilton--Jacobi--Bellman equations.
International Journal of Control. 92(10), 2263--2273 (2019).
https://doi.org/10.1080/00207179.2018.1436769

\bibitem{Zhou_2021}
Zhou, J.:
A notion of viscosity solutions to second-order Hamilton-Jacobi-Bellman equations with delays.
International Journal of Control (2021).
https://doi.org/10.1080/00207179.2021.1921279
\end{thebibliography}


\end{document}